\documentclass[12pt]{amsart}

\usepackage[latin1]{inputenc}
\usepackage{amssymb,amsmath,amsfonts, amsthm}
\usepackage{enumerate}
\usepackage{times}
\usepackage{bbding}

\usepackage{cite}

\newtheorem{theorem}{Theorem}[section]
\newtheorem{lemma}[theorem]{Lemma}
\newtheorem{proposition}[theorem]{Proposition}
\newtheorem*{assumption}{Induction Hypothesis}
\newtheorem{corollary}[theorem]{Corollary}
\numberwithin{equation}{section}

\newcommand{\fconvs}{{\mbox{\rm Conv}_{{\rm sc}}(\R^n)}} 
\newcommand{\fconvse}{{\mbox{\rm Conv}_{{\rm sc}}(\R)}} 
\newcommand{\fconvso}{{\mbox{\rm Conv}_{{\rm sc},0}(\R^n)}} 
\newcommand{\fconvsreg}{{\mbox{\rm Conv}_{{\rm rg}}(\R^n)}} 
\newcommand{\fconvx}{{\mbox{\rm Conv}(\R^n)}}
\newcommand{\fconvxE}{{\mbox{\rm Conv}(E)}}
\newcommand{\fconvxk}{{\mbox{\rm Conv}(\R^k)}}
\newcommand{\fconvf}{{\mbox{\rm Conv}(\R^n; \R)}}
\newcommand{\fconvfe}{{\mbox{\rm Conv}(\R; \R)}}
\newcommand{\fconvfk}{{\mbox{\rm Conv}(\R^k; \R)}}
\newcommand{\fconvfE}{{\mbox{\rm Conv}(E; \R)}}
\newcommand{\fconvfF}{{\mbox{\rm Conv}(F; \R)}}
\newcommand{\fconvsE}{{\mbox{\rm Conv}_{{\rm sc}}(E)}} 
\newcommand{\fconvsF}{{\mbox{\rm Conv}_{{\rm sc}}(F)}} 

\newcommand{\fconvfs}{{\mbox{\rm Conv}_{(0)}(\R^n;\R)}}
\newcommand{\fconvfsk}{{\mbox{\rm Conv}_{(0)}(\R^k;\R)}}
\newcommand{\fconvfsE}{{\mbox{\rm Conv}_{(0)}(E; \R)}}
\newcommand{\fconvfsEp}{{\mbox{\rm Conv}_{(0)}(E'; \R)}}
\newcommand{\fconvfsF}{{\mbox{\rm Conv}_{(0)}(F; \R)}}
\newcommand{\fconvfsFp}{{\mbox{\rm Conv}_{(0)}(F'; \R)}}

\newcommand{\infconv}{\mathbin{\Box}}
\newcommand{\sq}{\mathbin{\vcenter{\hbox{\rule{.3ex}{.3ex}}}}}
\newcommand{\ecup}{\sqcup}
\newcommand{\Ecup}{\bigsqcup}
\newcommand{\cP}{{\mathcal P}} 
\newcommand{\cK}{{\mathcal K}} 

\newcommand{\Lip}{\operatorname{Lip}}

\newcommand{\Abel}{\operatorname{\mathcal A}} 

\renewcommand{\star}{\circ}

\newcommand{\sn}{{\mathbb{S}^{n-1}}}
\newcommand{\sk}{{\mathbb S}^{k-1}} 
\newcommand{\Bn}{B^n}

\newcommand{\sln}{\operatorname{SL}(n)}

\newcommand{\Hess}{{\operatorname{D}}^2}
\renewcommand{\div}{\operatorname{div}}

\renewcommand{\O}{\operatorname{O}}
\newcommand{\SO}{\operatorname{SO}}

\DeclareMathOperator{\oZ}{\operatorname{Z}}
\DeclareMathOperator{\otZ}{\operatorname{\tilde Z}}

\newcommand{\oZZ}[2]{\operatorname{V}_{#1,#2}}
\newcommand{\ZZZ}[3]{\operatorname{V}_{#1,#2}^{(#3)}}

\newcommand{\R}{{\mathbb R}}
\newcommand{\N}{{\mathbb N}}

\newcommand{\spa}{\operatorname{span}}
\newcommand{\dom}{\operatorname{dom}}

\newcommand{\hm}{\mathcal H}
\newcommand{\epi}{\operatorname{epi}}
\renewcommand{\d}{\,\mathrm{d}}
\newcommand{\ind}{{\rm\bf I}}

\newcommand{\Had}[2]{D_{#1}^{#2}} 

\newcommand{\my}{\operatorname{env}} 

\begin{document}

\title{The Hadwiger Theorem on Convex Functions, I}

\author{Andrea Colesanti}
\address{Dipartimento di Matematica e Informatica ``U. Dini''
Universit\`a degli Studi di Firenze,
Viale Morgagni 67/A - 50134, Firenze, Italy}
\email{andrea.colesanti@unifi.it}

\author{Monika Ludwig}
\address{Institut f\"ur Diskrete Mathematik und Geometrie,
Technische Universit\"at Wien,
Wiedner Hauptstra\ss e 8-10/1046,
1040 Wien, Austria}
\email{monika.ludwig@tuwien.ac.at}

\author{Fabian Mussnig}
\address{School of Mathematical Sciences, Tel Aviv University, 69978 Tel Aviv, Israel}
\email{mussnig@gmail.com}

\date{}

\begin{abstract} 
A complete classification of all continuous, epi-translation and rotation invariant valuations on the space of super-coercive convex functions on $\R^n$ is established. The valuations obtained are functional versions of the classical intrinsic volumes. For their definition,  singular Hessian valuations are introduced.

\end{abstract}

\maketitle

\section{Introduction and Statement of Results}

Intrinsic volumes play a fundamental role in Euclidean geometry. For a convex body $K$ (that is, a non-empty, compact convex set) in $n$-dimensional Euclidean space, $\R^n$, and $j\in\{0,\dots,n\}$, the $j$th intrinsic volume, $V_j(K)$, measures, in a translation and rotation invariant way, the $j$-dimensional \lq\lq size\rq\rq\ of $K$. In particular, $V_n(K)$ is the $n$-dimensional volume of $K$ and $V_0(K)=1$ is the Euler characteristic. For $1\le j\le n-1$, the $j$th intrinsic volume can be defined in various ways, for example, via the Steiner formula, via integral geometric formulas, or, for a convex body $K$ with smooth boundary, as
\begin{equation*}
V_j(K)=\frac1{\omega_{n-j}} \int_{\sn} [\Hess h_K(y)]_{j} \d\hm^{n-1}(y),
\end{equation*}
where $h_K$ is the support function of $K$ (see Section~\ref{convex bodies} for the definition) and $\Hess h_K$ the Hessian matrix of $h_K$. Here, we write $[A]_j$ for the $j$th elementary symmetric function of the eigenvalues of a symmetric matrix $A$ and use the convention that $[A]_0=1$. We write $\hm^{n-1}$ for $(n-1)$-dimensional Hausdorff measure and $\omega_{k}$ for the $(k-1)$-dimensional Hausdorff measure  of the unit sphere, $\sk$, in $\R^k$. 
The $j$th intrinsic volume is homogeneous of degree $j$, that is, $V_j(\lambda K)=\lambda^j V_j(K)$ for all convex bodies $K$ in $\R^n$ and all $\lambda\ge 0$. The intrinsic volumes are continuous functionals on the space of convex bodies, $\cK^n$, in $\R^n$ equipped with the standard topology induced by the Hausdorff metric
(see \cite{Schneider:CB2} for information regarding  intrinsic volumes).

One of the most important results in geometry is Hadwiger's characterization theorem  \cite{Hadwiger:V} for intrinsic volumes.
Hadwiger's theorem classifies all continuous, translation and rotation invariant valuations
on convex bodies. Here a functional $\oZ: \cK^n\to \R$ is called a valuation if
\begin{equation}\label{val_bodies}
\oZ(K)+\oZ(L)= \oZ(K\cup L)+ \oZ(K\cap L)
\end{equation}
for every $K,L\in\cK^n$ such that $K\cup L\in \cK^n$. Recall that $\oZ$ is translation invariant if $\oZ(\tau K)=\oZ(K)$ for all $K\in\cK^n$ and translations $\tau$ on $\R^n$, and  it is rotation invariant if $\oZ(\vartheta K)=\oZ(K)$ for all $K\in\cK^n$ and $\vartheta\in\SO(n)$.

\begin{theorem}[Hadwiger \cite{Hadwiger:V}]\label{hugo}
A functional $\oZ:\cK^n\to \R$ is  a continuous, translation and rotation invariant valuation if and only if  there are constants $\zeta_0, \ldots, \zeta_n\in\R$ such that 
$$\oZ(K) =  \zeta_0 \,V_0(K)+\dots+\zeta_n\,V_n(K)$$
for every $K\in\cK^n$. 
\end{theorem}

\noindent
Hadwiger's theorem leads to effortless proofs of numerous results in integral geometry and geometric probability (see \cite{Hadwiger:V} and \cite{Klain:Rota}). It is the first culmination of a program initiated by Blaschke and the starting point of the geometric theory of valuations, where classification theorems for valuations invariant under various groups are fundamental (see \cite[Chapter~6]{Schneider:CB2}). Interesting new valuations keep arising (see, for example, \cite{HLYZ_acta, BLYZ_cpam, Ludwig:matrix}) and see \cite{Alesker99,Alesker01,AleskerFaifman, BernigFaifman2017,Bernig:Fu, Haberl_sln, Parapatits:Wannerer, LiMa,  Ludwig:Reitzner, Ludwig:Reitzner2,Ludwig:convex,Haberl:Parapatits_centro} for some recent classification results.
\goodbreak

Currently,  a geometric theory of valuations on function spaces is developed. A functional $\oZ$ defined on a  space of (extended) real-valued functions, $X$,  is called a \emph{valuation} if
\begin{equation*}\label{val_functions}
\oZ(u) +\oZ (v) = \oZ(u\vee v)+\oZ(u\wedge v)
\end{equation*}
for all $u,v\in X$ such that the pointwise maximum $u\vee v$ and the pointwise minimum $u\wedge v$ belong to $X$.
When $X$ is the space of indicator functions of convex bodies in $\R^n$, we recover the classical notion of valuation on convex bodies.
The first classification results were obtained for $L_p$ and Sobolev spaces,  for Lipschitz functions and for  definable functions (see, for example, \cite{LiMa, Tsang:Lp, Ludwig:SobVal, Ludwig:Fisher, ColesantiPagniniTradaceteVillanueva, BaryshnikovGhristWright}). 

Spaces of convex functions play a special role because of their close connection to convex bodies. Here classification results were obtained for $\sln$ invariant valuations in \cite{ColesantiLudwigMussnig, ColesantiLudwigMussnig17, Mussnig19, Mussnig_super}, the connection to valuations on convex bodies was explored in \cite{Alesker_cf} and first structural results were established in \cite{Colesanti-Ludwig-Mussnig-4, Knoerr1, Knoerr2}. The general space of extended real-valued convex functions on $\R^n$ is defined as
\begin{equation*}
\fconvx:=\{u\colon\R^n\to\R\cup\{+\infty\}\colon u\ \mbox{ is convex and lower semicontinuous},\ u\not\equiv+\infty\}.
\end{equation*}
It is equipped with the topology induced by  epi-convergence (see Section \ref{convex functions}), and the continuity of valuations defined on $\fconvx$ (and its subsets) will always be with respect to this topology.

The aim of this paper is to establish the Hadwiger theorem on convex functions.
A valuation $\oZ$ defined on (a subset of) the space $\fconvx$ is said to be \emph{epi-translation invariant} if $\oZ(u\circ \tau^{-1}+\alpha)=\oZ(u)$ for every translation $\tau$ on $\R^n$ and $\alpha\in\R$. It is said to be \emph{rotation invariant} if $\oZ(u\circ \vartheta^{-1})=\oZ(u)$ for every $\vartheta\in\SO(n)$.
We will identify a convex function with its epi-graph, 
$$\epi(u):=\{(x,t)\in\R^{n+1}: t\ge u(x)\}.$$ 
Note that an epi-translation invariant valuation on $\fconvx$ corresponds to a translation invariant valuation on the set of epi-graphs in the sense of (\ref{val_bodies}). We say that $\oZ$ is \emph{epi-homogeneous} of degree~$j$, if the associated valuation on epi-graphs is homogeneous of degree $j$.  It was shown in \cite{Colesanti-Ludwig-Mussnig-4} that a continuous and epi-translation invariant valuation on $\fconvx$ is always constant, and the corresponding statement  was proved on the space of coercive convex functions. Thus, we look at valuations on the smaller space of \emph{super-coercive convex functions},
\begin{equation*}
\fconvs:=\left\{u\in\fconvx\colon\lim_{\vert  x\vert \to+\infty}\frac{u(x)}{\vert  x\vert }=+\infty\right\},
\end{equation*}
where $\vert\cdot\vert$ is the Euclidean norm. For $n=1$, a complete classification of continuous and epi-translation invariant valuations was obtained in  \cite{Colesanti-Ludwig-Mussnig-4}. So, let $n\ge 2$.
For  $\zeta\in C_c([0,\infty))$ and  $j\in\{0,\dots, n\}$, it follows from results in \cite{Colesanti-Ludwig-Mussnig-3} that a valuation,   
defined for $u\in\fconvs\cap C_+^2(\R^n)$ (that is, with a positive definite Hessian matrix $\Hess u$) by
\begin{equation}\label{hessian}
\oZZ{j}{\zeta}(u):=  \int_{\R^n} \zeta(|\nabla u(x)|) [\Hess u(x)]_{n-j} \d x,
\end{equation}
can be extended to a continuous, epi-translation and rotation invariant valuation that is epi-homogeneous of degree $j$ on $\fconvs$. The valuation defined in (\ref{hessian}) is a {\em Hessian valuation}, and  the extension of the integral \eqref{hessian} uses the notion of {Hessian measures}. These measures play an important role in the study of a class of fully non-linear elliptic PDEs, the so-called Hessian equations, and in this context were introduced by Trudinger and Wang \cite{TrudingerWang1997, Trudinger:Wang1999}. In the case of convex functions, these measures have been studied, for example, in \cite{ColesantiHug2000} (see Section \ref{se:hessian_measures}).
\goodbreak

As we will show, the functionals $\oZZ{j}{\zeta}$ also extend to continuous valuations on $\fconvs$ for functions $\zeta\in C((0,\infty))$ with a certain type of singularity at the origin. We thus obtain {\em singular Hessian valuations}.
Let $C_b((0,\infty))$ be the set of continuous functions on $(0,\infty)$ with bounded support. For $0\leq j \leq n-1$, let
$$
\Had{j}{n}:=\Big\{\zeta\in C_b((0,\infty))\colon  \lim_{s\to 0^+} s^{n-j} \zeta(s)=0,  \lim_{s\to 0^+} \int_s^{\infty}  t^{n-j-1}\zeta(t) \d t \text{ exists and is finite}\Big\}.
$$
In addition, let $\Had{n}{n}$ be the set of functions $\zeta\in C_b((0,\infty))$ where $\lim_{s\to 0^+} \zeta(s)$ exists and is finite and set $\zeta(0):=\lim_{s\to 0^+} \zeta(s)$. Note that $\oZZ{0}{\zeta}$ is constant on $\fconvs$ for $\zeta\in\Had{0}{n}$ (see Section \ref{section valuations}).

\goodbreak
Our first main result establishes the existence of the extension of singular Hessian valuations for functions $\zeta\in\Had{j}{n}$.

\begin{theorem}\label{main one way} 
	For $j\in\{0,\dots,n\}$ and $\zeta\in\Had{j}{n}$, 
	there exists a unique, continuous, epi-translation and 	rotation invariant valuation 
	$\oZZ{j}{\zeta}\colon\fconvs\to\R$ such that 
	\begin{equation*}
	\oZZ{j}{\zeta}(u)=  \int_{\R^n} \zeta(|\nabla u(x)|) [\Hess u(x)]_{n-j} \d x
	\end{equation*}
	for every $u\in\fconvs\cap C^2_+(\R^n)$. 
\end{theorem}

\noindent
The proof of this result will use analytic tools developed within the theory of Hessian equations (see Section \ref{section prep}
for details). 
Given a singular function $\zeta$, the integral (\ref{hessian}) converges absolutely for sufficiently regular functions. For the general extension argument, Hessian valuations, which were introduced in \cite{Colesanti-Ludwig-Mussnig-3}, will be used, but singular functions are, in general, not integrable with respect to Hessian measures. For the extension, we will use the Moreau--Yosida approximation and the existence of a homo\-geneous decomposition for epi-translation invariant valuations that was established in \cite{Colesanti-Ludwig-Mussnig-4}.

\goodbreak

Our second main result shows that  the singular Hessian valuations that were introduced above span the space of continuous, epi-translation and rotation invariant valuations on $\fconvs$. 
Let $n\ge 2$.

\begin{theorem}
\label{thm:hadwiger_convex_functions}
A functional $\oZ:\fconvs \to \R$ is a continuous, epi-translation and rotation invariant valuation if and only if there exist  functions $\zeta_0\in\Had{0}{n}$, \dots, $\zeta_n\in\Had{n}{n}$  such that
\begin{equation*}
\oZ(u)= \oZZ{0}{\zeta_0}(u) +\dots +\oZZ{n}{\zeta_n}(u)
\end{equation*}
for every $u\in\fconvs$.
\end{theorem}

\noindent
These theorems show that the singular Hessian valuations $\oZZ{j}{\zeta}$ with $\zeta\in\Had{j}{n}$ are functional versions of the intrinsic volumes $V_j$ and that, from the point of view of geometric valuation theory, they are the canonical functional versions of intrinsic volumes.

The space of super-coercive convex functions is related to another subspace of $\fconvx$, formed by convex functions with finite values,
$$
\fconvf:=\big\{v\in\fconvx\colon v(x)<+\infty\,\text{ for all }\,x\in\R^n\big\}.
$$
A function $v$ belongs to  $\fconvf$ if and only if its standard conjugate or Legendre transform $v^*$ belongs to $\fconvs$ (see Section \ref{convex functions}). It was proved in \cite{Colesanti-Ludwig-Mussnig-3} that $\oZ$ is a continuous valuation on $\fconvf$ if and only if $\oZ^*\colon\fconvs\to\R$,  defined by
$$
\oZ^*(u):=\oZ(u^*),
$$
 is a continuous valuation on $\fconvs$. This fact permits us to transfer results  valid for valuations on $\fconvf$ to results for valuations on 
$\fconvs$ and vice versa.  
A valuation $\oZ$ on $\fconvf$ is called \emph{homogeneous} of degree $j\in\R$ if $\oZ(\lambda v)=\lambda^j \oZ(v)
$
for all $v\in\fconvf$ and $\lambda\geq 0$. It is \emph{dually epi-translation invariant} if 
$$
\oZ(v+\ell+\alpha)=\oZ(v)
$$
for every $v\in\fconvf$ and every linear function $\ell\colon\R^n\to\R$ and $\alpha\in\R$. The valuation $\oZ$ is dually epi-translation  invariant and homogeneous of degree $j$ on $\fconvf$ if and only if $\oZ^*$ is epi-translation  invariant and epi-homogeneous of degree $j$ on $\fconvs$. 
Using results on Hessian valuations for conjugate functions, we will obtain the following extension theorem from Theorem~\ref{main one way}.

\begin{theorem}\label{dual main one way} 
For $j\in\{0,\dots,n\}$ and $\zeta\in\Had{j}{n}$, 
the functional
$\oZZ{j}{\zeta}^*\colon\fconvf\to\R$ is a continuous, dually epi-translation and 	rotation invariant valuation  such that
\begin{equation}\label{dhessian}
	\oZZ{j}{\zeta}^*(v)=  \int_{\R^n} \zeta(|x|) [\Hess v(x)]_{j} \d x
\end{equation}
for every $v\in\fconvf\cap C^2_+(\R^n)$. 
\end{theorem}

\goodbreak
Theorem \ref{thm:hadwiger_convex_functions} has the following dual version. Let $n\ge 2$.

\begin{theorem}
\label{dthm:hadwiger_convex_functions}
A functional $\oZ:\fconvf \to \R$ is a continuous, dually epi-translation and rotation in\-va\-riant valuation if and only if there exist  functions $\zeta_0\in\Had{0}{n}$, \dots, $\zeta_n\in\Had{n}{n}$  such that
\begin{equation*}
\oZ(v)= \oZZ{0}{\zeta_0}^*(v) +\dots +\oZZ{n}{\zeta_n}^*(v)
\end{equation*}
for every $v\in\fconvf$.
\end{theorem}

\noindent In the proof of Theorem \ref{thm:hadwiger_convex_functions} and Theorem \ref{dthm:hadwiger_convex_functions}, we will use both the primal setting of epi-translation invariant valuations on $\fconvs$ and the dual setting of dually epi-translation invariant valuations on $\fconvf$. We follow the original approach of Hadwiger \cite{Hadwiger:V}. We introduce rotational epi-symmetrization to establish a classification of valuations that are epi-homogeneous of degree 1 on $\fconvs$ and induction on the dimension for the general case. The homogeneous decomposition of epi-translation invariant valuations established in \cite{Colesanti-Ludwig-Mussnig-4} will be an important tool in the proof.

An alternate proof of Hadwiger's theorem is due to Klain \cite{Klain95}. We will discuss elements of Klain's proof and how to apply Theorem \ref{thm:hadwiger_convex_functions} and Theorem \ref{dthm:hadwiger_convex_functions} in integral geometry in subsequent papers (see \cite{Colesanti-Ludwig-Mussnig-6,Colesanti-Ludwig-Mussnig-7,Colesanti-Ludwig-Mussnig-8}).

\section{Preliminaries}

We work in $n$-dimensional Euclidean space $\R^n$, with $n\ge 1$, endowed with the Euclidean norm $\vert \cdot\vert $ and the usual scalar product
$\langle \cdot,\cdot\rangle$. We write $e_1,\dots, e_n$ for the canonical basis vectors of $\R^n$ and use coordinates, $x=(x_1,\dots,x_n)$, for $x\in\R^n$ with respect to this basis. We write $\spa$ for linear span. Let $\Bn:=\{x\in\R^n: \vert x\vert \le 1\}$ be the Euclidean unit ball and $\sn$ the unit sphere in $\R^n$. We denote by $\kappa_n$ the $n$-dimensional volume of $\Bn$ and by $\omega_n$ the $(n-1)$-dimensional Hausdorff measure of $\sn$.

\subsection{Convex Sets}\label{convex bodies} A basic reference on convex sets is the book by Schneider \cite{Schneider:CB2}. Let $K\subset \R^n$ be a closed convex set that is non-empty and such that $K\ne \R^n$. Its {\em support function} $h_K:\R^n\to \R\cup \{+\infty\}$ is defined as
$$h_K(y):=\sup\nolimits_{x\in K}\langle x,y \rangle.$$
It is a one-homogeneous and convex function that determines $K$. If $K\subset \R^n$ is a convex body with boundary of class $C_+^2$ and $y\in\sn$, then the eigenvalues of its Hessian matrix, $\Hess h_K(y)$, are the principal radii of curvature of $K$ at $y$. Hence $[\Hess h_K(y)]_i$ is the $i$th elementary symmetric function of the principal radii of curvature of $K$ at $y$ (see \cite[Section 2.5]{Schneider:CB2}).

\subsection{Convex Functions}\label{convex functions}

We collect some properties of convex functions. Basic references are the books by Rockafellar \cite{Rockafellar} and Rockafellar \& Wets \cite{RockafellarWets}. 

Let $u\in\fconvx$. For  $t\in(-\infty,+\infty]$, its \emph{sublevel sets}
\begin{equation*}
\{u<t\}:=\{x\in\R^n:u(x)<t\},\quad\quad \{u\leq t\}:=\{x\in\R^n:u(x)\leq t\}
\end{equation*}
are convex, and the latter is closed by the lower semicontinuity of $u$. In addition, also its \emph{domain}, 
$$\dom (u):=\{x\in\R^n: u(x)<+\infty\}$$
is a convex set. We equip $\fconvx$  and its subspaces with the topology associated to epi-convergence.
Here a sequence $u_k$ of functions from $\fconvx$ is \emph{epi-convergent}  to $u\in\fconvx$ if the following conditions hold for all $x\in\R^n$:
\begin{itemize}
	\item[(i)] For every sequence $x_k$ that converges to $x$, we have $u(x) \leq \liminf_{k\to \infty} u_k(x_k)$.
			
	\item[(ii)] There exists a sequence $x_k$ that converges to $x$ such that $u(x) = \lim_{k\to\infty} u_k(x_k)$.
\end{itemize}
\goodbreak\noindent
Note that a sequence $v_k$ of functions from $\fconvf$ epi-converges to $v\in\fconvf$  if and only if $v_k$ converges pointwise to $v$ and uniformly on compact sets. We will just say that $v_k$ converges to $v$ in this case. We also require the following fact.
If $u_k$ is a sequence of functions from $\fconvx\cap C^1(\R^n)$ that epi-converges to $u\in\fconvx\cap C^1(\R^n)$, then $\nabla u_k$ converges uniformly to $\nabla u$ on compact sets (see \cite[Theorem 25.7]{Rockafellar}). 

\goodbreak

Given $u\in\fconvx$, the \emph{convex conjugate} $u^*: \R^n \to (-\infty,\infty]$ is defined by
$$
u^*(y):=\sup_{x\in\R^n}\big(\langle y,x \rangle - u(x) \big).
$$
As $u\in\fconvx$, also $u^*\in\fconvx$ and $u^{**}=u$. Moreover, $u\in\fconvs\cap C^2_+(\R^n)$ if and only if $u^*\in \fconvs\cap C^2_+(\R^n)$. Furthermore, $\fconvs\cap C^2_+(\R^n)$ is a dense subset of $\fconvf\cap C_+^2(\R^n)$. 
Given a subset $A\subset\R^n$, let $\ind_A\colon\R^n\to\R\cup\{+\infty\}$ denote the (convex) indicatrix function of $A$,
$$
\ind_A(x):=
\left\{
\begin{array}{lll}
\phantom{\,+}0\quad\quad&\mbox{if $x\in A$,}\\
+\infty\quad&\mbox{if $x\notin A$.}
\end{array}
\right.
$$
Note that for a convex body $K\subset \R^n$, we have $\ind_K\in\fconvs$.
Let $t\ge0$. We will often use the following pair of dual functions, 
\begin{equation}\label{ut def}
u_t(x):= t\,\vert x\vert +\ind_{\Bn}(x)
\end{equation}
and
\begin{equation}\label{vt def}
	v_t(x):=\begin{cases}
	0\quad\quad&\mbox{if $\,\vert x\vert \le t$,}\\
	\vert x\vert-t\quad&\mbox{if $\,\vert x\vert>t$.}
	\end{cases}
\end{equation}
Note that $u_t\in\fconvs$ and $v_t\in\fconvf$ while $u_t^*=v_t$.

For $u,v\in\fconvx$, let
$$(u\infconv v)(x) := \inf_{x=y+z} \big(u(y)+v(z)\big)$$
denote the \emph{infimal convolution} of $u$ and $v$ at $x\in\R^n$. Note, that
$$\epi (u\infconv v) = \epi u + \epi v,$$
where on the right, we have the Minkowski sum of the sets $\epi u$ and $\epi v$, and if $u\infconv v>-\infty$ pointwise, then
\begin{equation}\label{epi add dual}
(u\infconv v)^* = u^*+v^*.
\end{equation}
We define the \emph{epi-multiplication} of $u\in \fconvx$ and $\lambda>0$ by setting  
$$
\lambda \sq u(x):=\lambda\, u\left(
\frac x\lambda
\right)
$$ 
for $x\in\R^n$. Note that the epi-graph of $\lambda \sq u$ is obtained by  rescaling the epi-graph of $u$ by the factor $\lambda$.
It is easy to see that $u\in\fconvs$ implies $\lambda \sq u\in\fconvs$ for $\lambda\geq 0$.
A functional $\oZ:\fconvs\to \R$ is called \emph{epi-homogeneous} of degree $j\in\R$ if
$$
\oZ(\lambda \sq u)=\lambda^j \oZ(u)
$$
for all $u\in\fconvs$ and $\lambda> 0$. In the following, corresponding definitions will be used for $\fconvx$ and its subspaces.

Let $E$ and $F$ be orthogonal and complementary subspaces of $\R^n$ with $\dim E, \dim F\ge 1$. We can identify $E$ with $\R^k$ when $k=\dim E$ and write $\fconvxE$ for $\fconvxk$ in this case. We use corresponding definitions for subspaces of $\fconvxE$. For $x\in\R^n$, we write $x=(x_E,x_F)$ with $x_E\in E$ and $x_F\in F$.
If $u\in\fconvs$ is such that $\dom(u)\subset E$, then $v=u^*$ does not depend on $x_F$. The restriction of $v$ to $E$, which we denote by $v_E\in\fconvfE$, is equal to the Legendre transform (with respect to the ambient space $E$) of the restriction of $u$ to $E$, which we denote by $u_E\in\fconvsE$. In this case, we will identify $u$ with $u_E$ and $v$ with $v_E$. Hence, for $u_E\in\fconvsE$ and $u_F\in\fconvsF$, we can use this identification and define $u=u_E\infconv u_F\in\fconvs$ where we use the usual infimal convolution on $\R^n$. Setting $v_E=u_E^*$ and $v_F=u_F^*$ and using (\ref{epi add dual}), we obtain that
\begin{equation}\label{infconvEF}
(u_E\infconv u_F)^*= v_E^*+v_F^*.
\end{equation}

For $u\in\fconvx$ and $\lambda>0$, we define the \emph{Moreau--Yosida envelope}, $\my_\lambda u$, of $u$ as
$$\my_\lambda u := u \infconv \frac{|\cdot|^2}{2\lambda}.$$
If $u\in\fconvs$, then $\my_\lambda u \in\fconvs\cap C^1(\R^n)$, and in particular, $\dom (\my_\lambda u) = \R^n$ (see \cite[Theorem 2.26]{RockafellarWets}). If additionally, $u\in C^2(\R^n)$, then so is $\my_\lambda u$ (see \cite[Theorem 7.37]{RockafellarWets}). If $u_k\in \fconvs$ and $u_k$ epi-convergences to $u$, then $\my_\lambda u_k$ epi-convergences to $\my_\lambda u$.

\subsection{Hessian Measures}
\label{se:hessian_measures}
In this part, we briefly recall the notion of Hessian measures of convex functions. For a more detailed presentation, the reader is referred to \cite{ColesantiHug2000,Colesanti-Ludwig-Mussnig-3}.

We begin with recalling the definition of the subdifferential of  a function  $u\in\fconvx$. For  $x\in\R^n$, the subdifferential of $u$ at $x$ is defined by
$$
\partial u(x):=\{y\in\R^n\colon u(z)\ge u(x)+\langle y,z-x\rangle\;\text{ for } z\in\R^n\}.
$$\goodbreak
\noindent
We set
$$
\Gamma_u:=\{(x,y)\in\R^n\times\R^n\colon y\in\partial u(x)\},
$$
that is, $\Gamma_u$ is the graph of the subdifferential map. 

\goodbreak

The Hessian measures of $u$, that we denote by $\Theta^n_j(u,\cdot)$ for $j=0,\dots,n$, are non-negative measures defined on the Borel subsets of $\R^n\times\R^n$ that  can be introduced in the following way. Given a Borel subset $A$ of $\R^n\times\R^n$ and $s\ge0$, we consider the set,
$$
P_s(u,A):=\{x+sy\colon (x,y)\in\Gamma_u\cap A\}.
$$
Its $n$-dimensional Hausdorff measure is a polynomial in $s$; in other words, there exist $(n+1)$ non-negative coefficients $\Theta^n_0(u,A),\dots,\Theta^n_n(u,A)$, such that
\begin{equation*}
    \hm^n(P_s(u,A))=\sum_{j=0}^{n} s^j\Theta^n_{n-j}(u,A)
\end{equation*}
for every Borel set $A$ in $\R^n\times \R^n$ and $s\ge 0$. The previous relation defines the Hessian measures of $u$. They are locally finite Borel measures on $\R^n\times\R^n$.

We will frequently use marginals of Hessian measures. Given $j\in\{0,\dots,n\}$ and $u\in\fconvx$, we set
$$
\Phi^n_{j}(u,B):=\Theta^n_{n-j}(u,B\times\R^n)
\quad\quad\mbox{for every Borel subset $B\subseteq \R^n$},
$$
and 
$$
\Psi^n_{j}(u,B):=\Theta^n_{j}(u,\R^n\times B)
\quad\quad\mbox{for every Borel subset $B\subseteq \R^n$}.
$$
Note that
\begin{equation}\label{Hessian measures smooth case}
\d\Phi^n_j(u,x)=[\Hess u(x)]_{j}\d x
\end{equation}
for $u\in\fconvx\cap C^2(\R^n)$. If $B$ is a Borel set in $\R^n$ and $u\in\fconvx$, then
\begin{equation}
\label{eq:phi_0_n}
\Phi_0^n(u,B)=\hm^n(B)
\end{equation}
and
$$
\Phi_j^n(\lambda u,B)= \lambda^j \Phi_j^n(u,B)
$$
for $j\in\{0,\ldots,n\}$ and $\lambda>0$. 

Let $u_k$ be a sequence in $\fconvx$. If $u_k$ epi-converges to $u\in\fconvx$, then the sequence of measures $\Theta_j^n(u_k,\cdot)$ converges weakly to the measure $\Theta_j^n(u,\cdot)$ (see \cite[Theorem~7.3]{Colesanti-Ludwig-Mussnig-3}). In particular, we obtain the following result for $\zeta\in C_c(\R^n\times\R^n)$ with compact support in the second variable. If $B$ is a bounded Borel set in $\R^n$ and $\Theta_j^n(u, \partial B)=0$, then
\begin{equation}
\label{weak_cont}
\lim_{k\to\infty} \int_{B\times\R^n} \zeta(x, y)\d \Theta_j^n(u_k, (x,y))=\int_{B\times\R^n} \zeta(x, y)\d \Theta_j^n(u, (x,y)).
\end{equation}

The interplay of Hessian measures and convex conjugation is well understood. Let $u\in\fconvx$ and $j\in\{0, \dots, n\}$. It was shown in \cite[Theorem 8.2]{Colesanti-Ludwig-Mussnig-3} that 
$$\Theta_{j}^n(u,A)=\Theta_{n-j}^n(u^*,\hat{A})$$
for every Borel subset $A$ of $\R^n\times\R^n$ where $\hat{A}=\{(x,y)\in \R^n\times\R^n\colon (y,x)\in A\}$. As an immediate consequence, we obtain
\begin{equation}
\label{eq:int_zeta_u_is_int_zeta_v}
\int_{B} \zeta(y) \d\Psi_j^n(u,y) = \int_{B} \zeta(x) \d\Phi_j^n(u^*,x)
\end{equation}
for every $u\in\fconvs$ and  Borel subset $B\subseteq \R^n$, when $\zeta:\R^n\backslash\{0\}\to\R$ is such that one of the two integrals above and therefore both exist. In particular, 
\begin{equation}\label{substitution}
\int_{\R^n} \zeta(\nabla u(x)) [\Hess u(x)]_{n-j}\d x = \int_{\R^n} \zeta(x) [\Hess u^*(x)]_j \d x
\end{equation}
for every $u\in\fconvs\cap C^2_+(\R^n)$ and such $\zeta$.

\subsection{Valuations on Convex Functions}\label{section valuations}
The following homogeneous decomposition result was established in \cite[Theorem 1]{Colesanti-Ludwig-Mussnig-4}.
\begin{theorem}
\label{thm:mcmullen_cvx_functions}
If $\,\oZ:\fconvs\to\R$ is a continuous and epi-translation invariant valuation, then there are continuous, epi-translation invariant valuations $\oZ_i:\fconvs\to\R$ that are epi-homogeneous of degree $i$ such that $\oZ=\oZ_0+\cdots + \oZ_n$.
\end{theorem}

We say that a valuation $\oZ:\fconvs\to\R$ is epi-additive if
$$\oZ(\alpha\sq u \infconv \beta \sq v) = \alpha \oZ(u)+\beta \oZ(v)$$
for every $\alpha,\beta>0$ and $u,v\in\fconvs$. The following result is a consequence of  Theorem~\ref{thm:mcmullen_cvx_functions}.

\begin{lemma}[\!\!\cite{Colesanti-Ludwig-Mussnig-4}, Corollary 22]\label{epi-additve}
If $\,\oZ:\fconvs\to \R$ is a continuous and epi-translation invariant valuation that is epi-homogeneous of degree $1$, then $\oZ$ is epi-additive.
\end{lemma}

For $\zeta\in \Had{0}{n}$ and $u\in\fconvs\cap C^2_+(\R^n)$, the substitution $y=\nabla u(x)$ shows that
$$\oZZ{0}{\zeta}(u)=\int_{\R^n} \zeta(\vert \nabla u(x)\vert) [\Hess u(x)]_n\d x= \int_{\R^n} \zeta(\vert y\vert)\d y$$
does not depend on $u$. Hence its extension to $\fconvs$ is constant. We have the following more general result.

\begin{theorem}[\!\cite{Colesanti-Ludwig-Mussnig-4}, Theorem 25]
\label{thm:class_0-hom}
A functional $\oZ:\fconvs \to \R$ is a continuous and epi-translation invariant valuation that is epi-homogeneous of degree $0$, if and only if $\,\oZ$ is constant.
\end{theorem}

\goodbreak
For the other extremal degree, we have the following classification result.

\begin{theorem}[\!\cite{Colesanti-Ludwig-Mussnig-4}, Theorem 2]
\label{thm:class_n-hom}
A functional $\oZ:\fconvs \to \R$ is a continuous and epi-translation invariant valuation that is epi-homogeneous of degree $n$, if and only if there exists $\zeta\in C_c(\R^n)$ such that
$$\oZ(u)=\int_{\dom (u)} \zeta(\nabla u(x)) \d x$$
for every $u\in\fconvs$.	
\end{theorem}

As a simple consequence, we obtain the following result. Here, we say that $\oZ$ is \emph{reflection invariant} if $\oZ(u)=\oZ(u^-)$, where $u^-(x):=u(-x)$ for every $x\in\R^n$. 

\begin{corollary}
\label{cor:class_n-hom}
For $n\ge 2$, a functional $\oZ:\fconvs \to \R$ is a continuous, epi-translation and rotation invariant valuation that is epi-homogeneous of degree $n$, if and only if there exists $\zeta\in C_c([0,\infty))$ such that
$$\oZ(u)=\int_{\dom (u)} \zeta(|\nabla u(x)|) \d x$$
for every $u\in\fconvs$. For $n=1$, the same representation holds if we replace rotation invariance by reflection invariance.
\end{corollary}
\begin{proof}
Let $\oZ$ be given. By Theorem~\ref{thm:class_n-hom} there exists $\xi\in C_c(\R^n)$ such that
$$\oZ(u)=\int_{\dom (u)} \xi(\nabla u(x)) \d x$$
for every $u\in\fconvs$. Fix  $y\in\R^n$ and let $u(x)=\ind_{\Bn}(x)+\langle y,x \rangle$ for $x\in\R^n$. Note, that for $\vartheta\in\SO(n)$
$$u\circ \vartheta^{-1}(x)=\ind_{\Bn}(x)+\langle \vartheta^{-t}y, x\rangle.$$
Hence, using the $\SO(n)$ invariance of $\oZ$, we obtain
$$
\kappa_n\, \xi(y) = \oZ(u) = \oZ(u\circ \vartheta^{-1}) = \kappa_n\, \xi(\vartheta^{-t}y)
$$
for every $\vartheta\in\SO(n)$. Since $y\in\R^n$ was arbitrary, it follows that there exists $\zeta\in C_c([0,\infty))$ such that $\xi(y)=\zeta(|y|)$ for every $y\in\R^n$.
In case $n=1$, we simply need to choose $\vartheta(x)=-x$ in the last argument.
\end{proof}

Combining this corollary with Theorem \ref{thm:class_0-hom} gives the following result.

\begin{corollary}
\label{cor:1}
A functional $\,\oZ:\fconvse \to \R$ is a continuous, epi-translation and reflection invariant valuation, if and only if there exist $\zeta_0\in\Had{0}{1}$ and $\zeta_1\in\Had{1}{1}$ such that
$$\oZ(u)=\oZZ{0}{\zeta_0}(u) +\oZZ{1}{\zeta_1}(u)$$
for every $u\in\fconvs$. 
\end{corollary}

\noindent
We remark that this implies that Theorem \ref{dthm:hadwiger_convex_functions} is also true in the one-dimensional case if we use the additional assumption of reflection invariance.

All of the previous results have dual versions. The following result is Theorem 4 from \cite{Colesanti-Ludwig-Mussnig-4}.

\begin{theorem}
\label{thm:mcmullen_cvx_functions dual}
If $\,\oZ:\fconvf\to\R$ is a continuous and dually epi-translation invariant valuation, then there are continuous and dually epi-translation invariant valuations $\oZ_i:\fconvf\to\R$ that are homogeneous of degree $i$ such that $\oZ=\oZ_0+\cdots + \oZ_n$.
\end{theorem}

\begin{theorem}[\!\cite{Colesanti-Ludwig-Mussnig-4}, Theorem 25]
\label{thm:class_0-hom dual}
A functional $\oZ:\fconvf \to \R$ is a continuous and dually epi-translation invariant valuation that is homogeneous of degree $0$, if and only if $\oZ$ is constant.
\end{theorem}

\begin{theorem}[\!\cite{Colesanti-Ludwig-Mussnig-4}, Theorem 5]
\label{thm:class_n-hom dual}
A functional $\oZ:\fconvf \to \R$ is a continuous and dually epi-translation invariant valuation that is homogeneous of degree $n$, if and only if there exists $\zeta\in C_c(\R^n)$ such that
$$\oZ(v)=\int_{\R^n} \zeta(x) \d \Phi^n_n(v,x)$$
for every $v\in\fconvf$.	
\end{theorem}

\begin{corollary}
\label{cor:class_n-hom dual}
For $n\ge 2$, a functional $\oZ:\fconvf \to \R$ is a continuous, dually epi-translation and rotation invariant valuation that is homogeneous of degree $n$, if and only if there is $\zeta\in C_c([0,\infty))$ such that
$$\oZ(v)=\int_{\R^n} \zeta(|x|) \d \Phi^n_n(v,x)$$
for every $v\in\fconvs$. For $n=1$, the same representation holds if we replace rotation invariance by reflection invariance.
\end{corollary}

The dual to Corollary \ref{cor:1} can be written in the following way. The integrals are well-defined by the definition of $\Had{0}{1}$ together with \eqref{eq:phi_0_n}, as well as, Theorem~\ref{thm:class_n-hom dual}.

\begin{corollary}
\label{cor:1d}
If $\oZ:\fconvfe \to \R$ is a continuous, dually epi-translation and reflection invariant valuation, then there exist $\zeta_0\in\Had{0}{1}$ and $\zeta_1\in\Had{1}{1}$ such that
$$
\oZ(v)=\int_{\R} \zeta_0(\vert x\vert)\d \Phi_0^1(v,x) + \int_{\R} \zeta_1(\vert x\vert)\d\Phi_1^1(v,x)
$$
for every $v\in\fconvfe$. 
\end{corollary}

The following was shown in \cite[Theorem 17]{Colesanti-Ludwig-Mussnig-4}.

\begin{lemma}
\label{le:hessian_val_zeta_cont_comp_supp}
Let $\zeta\in C(\R\times\R^n\times\R^n)$ have compact support with respect to the second variable. For $j\in\{0,1,\ldots,n\}$,
$$\oZ(v)=\int_{\R^n\times\R^n}\zeta(v(x),x,y)\d \Theta_j^n(v,(x,y))$$
is well-defined for every $v\in\fconvf$ and defines a continuous valuation on $\fconvf$.
\end{lemma}

\goodbreak
We conclude this part with a result that will be repeatedly used throughout this paper. It is a direct consequence of \eqref{eq:int_zeta_u_is_int_zeta_v} and Lemma~\ref{le:hessian_val_zeta_cont_comp_supp}.

\begin{lemma}\label{fine zeta}
    For $\zeta\in C_c([0,\infty))$ and $j\in\{0,\dots,n\}$, the functional $\oZ\colon\fconvs\to\R$, defined by
    $$
    \oZ(u)=\int_{\R^n}\zeta(|y|)\d\Psi^n_j(u,y),
    $$
    is a continuous, epi-translation and rotation invariant valuation that is epi-homogeneous of degree $j$.
\end{lemma}

\subsection{The Abel transform.}
For $\zeta\in C_b((0,\infty))$, define its \emph{Abel transform} for $t>0$ as
$$\Abel \zeta (t) :=2 \int_t^{\infty}\frac{s\,\zeta(s)\d s}{\sqrt{s^2-t^2}}=\int_{-\infty}^\infty \zeta(\sqrt{s^2+t^2})\d s.$$
If $\xi\in C_b^1((0,\infty))$, then the inverse transform is given by
$$\Abel^{-1} \xi (s) = -\frac1\pi \int_s^\infty \frac{\xi'(t) \d t}{\sqrt{t^2-s^2}}.$$
In particular, $ \Abel \zeta \equiv 0$ implies that $\zeta\equiv 0$.   More generally, we have the following fact. 
\begin{lemma}
	\label{lemma Abel}
	Let $\zeta\in C_b((0,\infty))$ and  $k\in\N\cup\{0\}$. If
	\begin{equation}\label{Abel}
	\int_0^{\infty}\zeta(\sqrt{r^2+t^2})r^k\d r=0
	\end{equation}
for every $t>0$,	then $\zeta\equiv 0$. 
\end{lemma}
\begin{proof}
By the definition of the Abel transform, we have
$$\mathcal{A}^{k+1} \zeta(t) = \int_{\R^{k+1}} \zeta(\sqrt{|x|^2+t^2}) \d x = (k+1) \kappa_{k+1} \int_0^{+\infty} \zeta(\sqrt{r^2+t^2})r^k \d r$$
for every $t>0$. Thus, it follows from \eqref{Abel} that $\mathcal{A}^{k+1}\zeta\equiv 0$, which implies that $\zeta\equiv 0$.
\end{proof}

\subsection{The Classes $\Had{j}{n}$} Let $j\in\{1,\dots,n-1\}$ and let $\zeta\in \Had{j}{n}$. We will associate with $\zeta$ two functions $\eta,\rho\in C_c([0,\infty))$ where
\begin{equation}
\label{eq:def_eta}
\eta(t)=\int_t^\infty s^{n-j-1}\zeta(s)  \d s
\end{equation}
and
\begin{equation}
\label{eq:def_rho}
\rho(t)=t^{n-j}\zeta(t)+(n-j)\eta(t)
\end{equation}
for every $t>0$. Observe that it follows from the definition of $\Had{j}{n}$ that $\eta(0)=\lim_{t\to 0^+}\eta(t)$ exists and is finite and that $\rho(0)=\lim_{t\to 0^+} \rho(t)=(n-j)\eta(0)$.

\goodbreak
The following result gives a geometric interpretation of the function $\rho$ using the function $u_t$ and $v_t$ defined in (\ref{ut def}) and (\ref{vt def}).

\begin{lemma}
	\label{le:calc_ind_bn_tx_theta_i}
	Let  $j\in\{1,\dots, n-1\}$ and $\zeta\in\Had{j}{n}$. For every $t \geq 0$,
	\begin{equation*}
	\int_{\R^n} \zeta(\vert y\vert)\d\Psi_j^{n}(u_t,y)=\int_{\R^n} \zeta(|x|) \d\Phi_j^n(v_t,x)=\kappa_n\binom{n}{j}\rho(t).
    \end{equation*}
\end{lemma}

\begin{proof}
    Recall that $v_t=(u_t)^*$. Hence, the first equality follows from  \eqref{eq:int_zeta_u_is_int_zeta_v}.	
	For $t>0$, all Hessian measures of $v_t$, with the exception of the measure $\Theta^n_n(v_t, \cdot)$, vanish in the set
	$\{x\in\R^n\colon \vert x\vert<t\}\times \R^n$. In particular,  we have $\Theta^n_{j}(v_t,B)=0$
	for every Borel subset $B$ of $\{x\in\R^n\colon \vert x\vert <t\}\times\R^n$. Moreover, the function $v_t$ is of class $C^2$ in the set $\{x\in\R^n\colon \vert x\vert>t\}$
	and the Hessian matrix of $v_t$ at a point $x$ in this set has $(n-1)$ eigenvalues equal to $1/\vert x\vert$ and the last eigenvalue equal to zero. Hence
	$$
	[\Hess v_t (x)]_j=\binom{n-1}{j}\frac{1}{\vert x\vert^j}
	$$
	for $\vert x\vert>t$.
	The set $t\,\sn$ is the set of singular points for $v_t$. For $x\in t\,\sn$, we have
	$$
	\partial v_t(x)=\Big\{r\,\frac{x}{\vert x\vert}: r\in [0,1]\Big\}.
	$$
	Consequently, $P_s(v_t, t\,\sn\times \R^n)= \{(t+r)x: x\in \sn, r\in[0,s]\}$ for $s\geq 0$. Thus,
	$$\hm^n\big(P_s(v_t,t\,\sn\times \R^n)\big)= \kappa_n \big((t+s)^n-t^n\big)$$
	and therefore 
	$$\Phi_{j}^n(v_t,t\, \sn)= \Theta^n_{n-j}(v_t, t\,\sn\times\R^n)=\kappa_n\binom{n}{j}\,t^{n-j}.$$
	Summing up, we obtain
	\begin{align*}
	\int_{\R^n} \zeta(\vert x\vert) \d\Phi_j^n(v_t,x)&= \kappa_n \binom{n}{j} t^{n-j} \zeta(t) + \binom{n-1}{j} \int_{\{x\in\R^n\colon \vert x \vert>t\}} \zeta(\vert x \vert)\frac{1}{\vert x\vert^j} \d x\\
	&= \kappa_n \binom{n}{j} t^{n-j} \zeta(t) + n \kappa_n \binom{n-1}{j}  \int_t^{+\infty} s^{n-j-1} \zeta(s) \d s,
	\end{align*}
	which implies the result for $t>0$. Note that the set of singular points of $v_0$ is $\{0\}$ and that $\partial v_0(0)=\Bn$. Hence $\Theta^n_{n-j}(v_0, \{0\}\times \Bn)=0$ which implies that the result holds for $t=0$.
\end{proof}

\section{Singular Hessian Valuations}\label{se:singular_hessian_valuations}
This section is devoted to the proof of Theorem \ref{main one way} and Theorem \ref{dual main one way}. We start with  functions in $\fconvf\cap C^2_+(\R^n)$ and show that on this class of functions, the valuations from (\ref{dhessian}) are well-defined and finite for $\zeta\in\Had{j}{n}$ with $j\in\{0,\dots,n\}$. The dual result implies that (\ref{hessian}) is well-defined and finite for such $\zeta$ on $\fconvs\cap C^2_+(\R^n)$.  The next step is an extension to functions in $\fconvs\cap C^1(\R^n)$. For this, we collect results related to the theory of Hessian equations and establish uniform estimates for regular functions. In the last step of the proof, we use Moreau--Yosida envelopes and the polynomiality of epi-translation invariant valuations, which was proved in \cite{Colesanti-Ludwig-Mussnig-4}, to extend the valuations to general functions in $\fconvs$. In the last part, we apply Theorem \ref{main one way} to establish a representation formula for $\oZZ{j}{\zeta}$ on a certain class of functions.

\subsection{The Smooth Case}

Let $\fconvfs$ denote the set of functions $v\in\fconvf$ that are of class $C^2$ in a neighborhood of the origin. 

\begin{lemma}\label{lem:existence_v_smooth_origin} Let $j\in\{1,\dots, n\}$ and $\zeta\in \Had{j}{n}$. If $\,v\in\fconvfs$, then
\begin{equation}
\label{eq:int_lem_existence_v_smooth_origin}
\int_{\R^n}\big\vert\zeta(|x|)\big\vert \d\Phi^n_{j}(v,x)
\end{equation}
is well-defined and finite.
\end{lemma}

\begin{proof}
Let $\delta>0$ be such that $v$ is of class $C^2$ on $\{x\in\R^n: \vert x\vert <\delta\}$. We rewrite \eqref{eq:int_lem_existence_v_smooth_origin} as
$$
\int_{\{x\colon |x|<\delta/2\}}\big\vert \zeta(|x|) \big\vert\d\Phi^n_{j}(v,x)+\int_{\{x\colon|x|\ge\delta/2\}}\big\vert\zeta(|x|) \big\vert\d\Phi^n_{j}(v,x).
$$
The second term is bounded, as $\zeta\in C_b((0,\infty))$
and $\Phi^n_{j}(v,\cdot)$ is locally finite. Concerning the first term,
let $\gamma>0$ be such that $[\Hess v(x)]_j\le \gamma$ for every $x$ such that $|x|\le \delta/2$. Using (\ref{Hessian measures smooth case}), we get
\begin{eqnarray*}
\int_{\{x\colon |x|<\delta/2\}}\big\vert\zeta(|x|)\big\vert\d\Phi^n_{j}(v,x)&=&\int_{\{x\colon|x|<\delta/2\}}\big\vert\zeta(|x|)\big\vert[\Hess v(x)]_j\d x\\
&\le& \gamma\int_{\{x\colon|x|<\delta/2\}}\big\vert\zeta(|x|)\big\vert\d x\\
&=&\omega_n\gamma\int_0^{\delta/2}r^{n-1}\big\vert\zeta(r)\big\vert\d r.
\end{eqnarray*}
As $\zeta\in \Had{j}{n}$, the function $r^{n-1}\big\vert\zeta(r)\big\vert$ can be extended to $r=0$ as a continuous function. Hence the last integral is finite.
\end{proof}
\goodbreak

\noindent
We remark that as a consequence of this lemma, we obtain that \eqref{dhessian} in Theorem~\ref{dual main one way} holds for every $v\in\fconvf\cap C^2(\R^n)$.
A special case of the dual result of Lemma~\ref{eq:int_lem_existence_v_smooth_origin} is the content of the following lemma.

\begin{lemma}\label{existence in the smooth case} Let $j\in\{1,\dots, n\}$ and $\zeta\in \Had{j}{n}$. If $\,u\in\fconvs\cap C^2_+(\R^n)$, then
$$
	\int_{\R^n} \zeta(|\nabla u(x)|) [\Hess u(x)]_{n-j} \d x
$$
is well-defined and finite.
\end{lemma}

\begin{proof}
Note that $u^*\in \fconvfs$ for $u\in\fconvs\cap C^2_+(\R^n)$. Hence, the statement follows from Lemma~\ref{lem:existence_v_smooth_origin} combined with (\ref{substitution}).
\end{proof}

\subsection{Preparatory Results for Theorem \ref{main one way}}\label{section prep}

We introduce the following class of functions,
$$\fconvsreg:=\{u\in\fconvs: u\ \mbox{ is regular, }\ 0=u(0)< u(x)\ \mbox{ for every }\ x\neq 0\},$$
where $u$ is \emph{regular} if	$u\in C^2(\R^n)$ and the boundary of its sublevel sets $\{u\le t\}$ is  of class $C^2$ with positive Gaussian curvature for every $t>0$.

Let $u\in\fconvsreg$. Note that $\nabla u(x)\ne0$ for every $x\ne0$ and that, for $t>0$, the sublevel set $\{u\le t \}$
is a convex body with non-empty interior and
$$
\partial \{u\le t \}=\{u=t\}.
$$
Given $x\ne 0$ and $j\in\{0,\dots,n-1\}$, we denote by $\tau_j(u,x)$ the $j$th elementary symmetric function of the principal curvatures at $x$ of $\{u\le t\}$ for $t=u(x)$. For every $t\ge0$, we denote by $h(u,\cdot,t)\colon\sn\to\R$, the support function of $\{u\le t \}$. If $t>0$, the regularity properties of the sublevel set $\{u\le t \}$ imply that $h(u,\cdot,t)\in C^2(\sn)$. For $y\in\sn$, let $\sigma_j(u,y,t)$ denote the $j$th elementary symmetric function of the principal radii of curvature of $ \{u= t \}$ at $\nu_t^{-1}(y)$, where $\nu_t\colon\{u=t\}\to\sn$ denotes the Gauss map of $\{u\le t \}$ and $\nu_t^{-1}: \sn\to\{u=t\}$ its inverse.

\subsubsection{Hessian Operators and Reilly-Type Lemmas}
Let $A=(a_{ij})$ be a symmetric $n\times n$ matrix and $k\in\{0,\dots,n\}$. We denote by $[A]_k$ the $k$th elementary symmetric function of the eigenvalues of $A$, that is, 
$$
[A]_k:=\sum_{1\le i_1<i_2<\dots<i_k\le n}\lambda_{i_1}\cdots\lambda_{i_k}
$$
(with the usual convention $[A]_0\equiv 1$), where $\lambda_1,\dots,\lambda_n$ are the eigenvalues of $A$. For $l,m\in\{1,\dots, n\}$, we set
\begin{equation}\label{cofactor}
[A]_k^{lm}:=\frac{\partial [A]_k}{\partial a_{lm}}.
\end{equation}
The $n\times n$ matrix formed by the entries $[A]_k^{lm}$ is called the \emph{cofactor matrix} of order $k$ of the matrix $A$. When $k=n$, this is the usual cofactor matrix. From the homogeneity of the map $A\mapsto[A]_k$, we deduce
\begin{equation}\label{fc1}
[A]_k^{lm}a_{lm}=k\,[A]_k
\end{equation}
(here and throughout this section, we adopt the convention of summation over repeated indices). 
For a $n\times j$ matrix $(b^{lm})$ with elements of class $C^1$, we define
$$\div_m((b^{lm}))$$
as the divergence of the vector field whose components are the elements of the $l$th column of the matrix $(b^{lm})$.

\goodbreak
Let $u\in C^2(\R^n)$. For $k\in\{0,\dots,n\}$, the $k$th \emph{Hessian operator} applied to $u$ and evaluated at a point $x\in\R^n$ is given by
$$
[\Hess u(x)]_k.
$$
The following result is due to Reilly \cite[Proposition 2.1]{Reilly}. 

\begin{lemma}\label{fc lemma 1} Let  $k\in\{0,1,\dots,n\}$ and $l\in\{1,\dots,n\}$. Then
	$$
	\div_m([\Hess u(x)]^{lm}_k)=0
	$$
	for every $u\in C^3(\R^n)$ and $x\in\R^n$. 
\end{lemma}

We will also need a corresponding result for functions defined on the unit sphere. Let $h\in C^2(\sn)$. Given $x\in\sn$, for $i,j\in\{1,\dots,n-1\}$ we  denote by $h_i$ and $h_{ij}$ the first and second covariant derivatives of $h$ with respect to a local orthonormal frame defined in a neighborhood of $x$. Let $\delta_{ij}$ be the usual Kronecker symbols. 
The following result was proved  in \cite{Cheng-Yau} in the case $m=n-1$ and then extended to the general case in \cite[Lemma 3.1]{Colesanti-Saorin}. 

\begin{lemma}\label{fc lemma 2} Let  $k\in\{0,\dots, n-1\}$ and $l\in\{1,\dots, n-1\}$. Then
	$$
	\div_m\left(
	[h(x)_{lm}+h(x)\delta_{lm}]^{lm}_k
	\right)=
	\sum_{m=1}^{n-1}\left(
	[h(x)_{lm}+h(x)\delta_{lm}]^{lm}_k
	\right)_m=0
	$$
	for every $h\in C^3(\sn)$ and $x\in\sn$, where $\left([h(x)_{lm}+h(x)\delta_{lm}]^{lm}_k\right)_m$ denotes the derivative of $[h(x)_{lm}+h(x)\delta_{lm}]^{lm}_k$ with respect to the $m$th variable of a local orthonormal frame defined in a neighborhood of $x$. 
\end{lemma}

\goodbreak
As a consequence, we obtain the following facts.

\begin{corollary}\label{towards self adjoint} Let $k\in\{0,\dots, n-1\}$. If $f, g, h\in C^2(\sn)$, then
	$$
	\int_{\sn}f\ [h_{lm}+h\delta_{lm}]^{lm}_k
	g_{lm}\d\hm^{n-1}=
	-\int_{\sn}[h_{lm}+h\delta_{lm}]^{lm}_k
	f_{l}\,g_m\d\hm^{n-1}.
	$$
\end{corollary}
\begin{proof} When $h\in C^3(\sn)$, the statement follows from the divergence theorem on the sphere and Lemma \ref{fc lemma 2}:
	\begin{eqnarray*}
		0&=&\int_{\sn}\sum_{m=1}^{n-1}\left(f\ [h_{lm}+h\delta_{lm}]^{lm}_kg_{l}\right)_m\d\hm^{n-1}\\
		&=&\int_{\sn}f\ [h_{lm}+h\delta_{lm}]^{lm}_k
		g_{lm}\d\hm^{n-1}+\int_{\sn}[h_{lm}+h\delta_{lm}]^{lm}_k
		f_{l}g_m\d\hm^{n-1}.
	\end{eqnarray*}
	The general case is obtained by approximation.
\end{proof}

\begin{corollary}\label{self adjoint} Let $k\in\{0,\dots, n-1\}$. If $f, g, h\in C^2(\sn)$, then
	$$
	\int_{\sn}f\ [h_{lm}+h\delta_{lm}]^{lm}_k
	(g_{lm}+g\delta_{lm})\d\hm^{n-1}=
	\int_{\sn}g\ [h_{lm}+h\delta_{lm}]^{lm}_k
	(f_{lm}+f\delta_{lm})\d\hm^{n-1}.
	$$
\end{corollary}

\subsubsection{Geometric Statements}
The following result is Proposition 2.4 from \cite{Longinetti-Salani}. 
\begin{lemma}
	\label{le:frac_1_nabla_u_h_t}
	For $u\in\fconvsreg$ the function $h(u,\cdot,\cdot)$ is of class $C^2(\sn\times (0,\infty))$ and, for every  $(y,t)\in\sn\times (0,\infty)$,
	$$
	\frac1{|\nabla u(\nu_t^{-1}(y))|}= \frac{\partial}{\partial t} h(u,y,t).
	$$
\end{lemma}

\goodbreak
The next result follows from formulas (4.9), (5.55) and (5.56) in \cite{Schneider:CB2}.
\begin{lemma}
	\label{le:tau_2}
	Let $1\leq i \leq n-1$ and $t>0$. For $u\in\fconvsreg$, 
	$$
	\int_{\{u= t \}} \tau_{n-i-1}(u,x) \d \hm^{n-1}(x)= \alpha\, V_{i}(\{u\le t \}),
	$$ 
	and
	$$
	\int_{\sn} h(u,y,t) \sigma_{i-1}(u,y,t) \d \hm^{n-1}(y) = \alpha\, V_i(\{u\le t \}),
	$$
    where $\alpha$ is a positive constant depending only on $n$ and $i$.
\end{lemma}

\begin{lemma}
	\label{le:tau_1}
	Let  $1\leq i \leq n-1$ and $0<t_1<t_2$. For $u\in\fconvsreg$,
	$$\int_{\{t_1<u\le t_2\}} \tau_{n-i}(u,x) \d x = \alpha \big( V_{i}(\{u\le t_2\}) - V_{i}(\{u\le t_1\}) \big),$$
	where $\alpha$ is a positive constant depending only on $n$ and $i$.
\end{lemma}
\begin{proof}
	Note that $\nabla u(x)\ne0$ for every $x\in\{t_1<u\le t_2\}$.
	By the coarea formula, we have
	$$
	\int_{\{t_1<u\le t_2\}}\tau_{n-i}(u,x)\d x=\int_{t_1}^{t_2}\int_{\{u= t \}}\frac1{|\nabla u(x)|}\,\tau_{n-i}(u,x)\d\hm^{n-1}(x)\d t.
	$$
	Using the change of variable $y=\nu_t(x)$ in the inner integral and Lemma~\ref{le:frac_1_nabla_u_h_t}, we get 
	\begin{align*}
	\int_{\{u= t \}}\frac1{|\nabla u(x)|}\,\tau_{n-i}(u,x)\d\hm^{n-1}(x)&=
	\int_{\sn}\frac1{|\nabla u(\nu_t^{-1}(y))|}\sigma_{i-1}(u,y,t)\d\hm^{n-1}(y)\\
	&=\int_{\sn}\frac{\partial}{\partial t} h(u,y,t)\, \sigma_{i-1}(u,y,t)\d\hm^{n-1}(y).
	\end{align*}
	Next, we prove that
	\begin{equation}\label{fc3}
	\frac{\d}{\d t}\int_{\sn} h(u,t,y)\,\sigma_{i-1}(u,y,t)\d\hm^{n-1}(y)=
	i\int_{\sn} \frac{\partial}{\partial t} h(u,y,t)\,\sigma_{i-1}(u,y,t)\d\hm^{n-1}(y).
	\end{equation}
	We have 
	\begin{equation}\label{equal}
	\sigma_{i-1}(u,y,t)=[h(u,y,t)_{lm}+h(u,y,t)\delta_{lm}]_{i-1}
	\end{equation}
	(see, for example, \cite[Section 2.5]{Schneider:CB2}).
	We start with the formula,
	\begin{eqnarray}\label{fc4}
	\frac{\d}{\d t}\int_{\sn} h(u,y,t)\,\sigma_{i-1}(u,y,t)\d\hm^{n-1}(y)&=&
	\int_{\sn} \frac{\partial}{\partial t}h(u,y,t)\,\sigma_{i-1}(u,y,t)\d\hm^{n-1}(y)\\
	&&+\int_{\sn} h(u,y,t)\,\frac{\partial\sigma_{i-1}}{\partial t}(u,y,t)\d\hm^{n-1}(y).\nonumber
	\end{eqnarray}
	By \eqref{equal}, \eqref{cofactor}, Corollary \ref{self adjoint}, and \eqref{equal},
	\begin{align*}
	\int_{\sn}& h(u,y,t)\frac{\partial}{\partial t}\,\sigma_{i-1}(u,y,t)\d\hm^{n-1}(y)\\
	&=\int_{\sn} h(u,y,t)[h(u,y,t)_{lm}+h(u,y,t)\delta_{lm}]_{i-1}^{lm}\big (\big(\frac{\partial}{\partial t}h (u,y,t)\big)_{lm}+\frac{\partial}{\partial t} h(u,y,t)\,\delta_{lm}\big)\d\hm^{n-1}(y)\\
	&=\int_{\sn} \frac{\partial}{\partial t}h(u,y,t)\,[h(u,y,t)_{lm}+h(u,y,t)\delta_{lm}]_{i-1}^{lm}\big(h(u,y,t)_{lm}+h(u,y,t)\delta_{lm}\big)\d\hm^{n-1}(y)\\
	&=(i-1)\int_{\sn} \frac{\partial}{\partial t}h(u,y,t)[h(u,y,t)_{lm}+h(u,y,t)\delta_{lm}]_{i-1}\d\hm^{n-1}(y)\\
	&=(i-1)\int_{\sn} \frac{\partial}{\partial t}h(u,y,t)\sigma_{i-1}(u,y,t)\d\hm^{n-1}(y).
	\end{align*}
	The last chain of equations and \eqref{fc4} imply \eqref{fc3}. Hence, by the latter relation,
	\begin{align*}
	\int\limits_{\{t_1<u\le t_2\}}&\tau_{n-i}(u,x)\d x\\[-6pt]
	&=\int_{t_1}^{t_2}\int_{\sn} \frac{\partial}{\partial t}h(u,y,t)\,\sigma_{i-1}(u,y,t)\d\hm^{n-1}(y)\d t\\
	&=\frac1i\int_{t_1}^{t_2}\frac{\d}{\d t}\int_{\sn} h(u,t,y)\,\sigma_{i-1}(u,y,t)\d\hm^{n-1}(y)\d t\\
	&=\frac1i\Big(\int\nolimits_{\sn} h(u,y,t_2)\,\sigma_{i-1}(u,y,t_2)\d\hm^{n-1}(y)
	-\int\nolimits_{\sn} h(u,y,t_1)\,\sigma_{i-1}(u,y,t_1)\d\hm^{n-1}(y)\Big)\\
	&=\alpha\,\big(V_{i}(\{u\le t_2\})-V_{i}(\{u\le t_1\})\big),
	\end{align*}
	where we have used Lemma~\ref{le:tau_2} in the last equation.
\end{proof}

\subsubsection{Integration by Parts}

We start with a lemma which can be found in \cite{BrandoliniNitschSalaniTrombetti_ARMA_2008} (see formula (11)). We write $u_k$ for the partial derivative of $u$ with respect to $x_k$.

\begin{lemma}[\!\cite{BrandoliniNitschSalaniTrombetti_ARMA_2008}]\label{lemma BNST} 
	Let $i\in\{1,\dots,n-1\}$ and $u\in\fconvsreg$. Then
	\begin{equation*}
	[\Hess u(x)]_{n-i}=|\nabla u(x)|^{n-i}\ \tau_{n-i}(u,x)+
	\frac{[\Hess u(x)]_{n-i}^{lm}u_k(x)u_l(x)u_{km}(x)}{|\nabla u(x)|^2}
	\end{equation*}
	for every $x\ne0$.
\end{lemma}

\goodbreak
We will also need the following lemma for which we refer, for instance, to \cite{Reilly}.

\begin{lemma}\label{lemma Reilly} Let $i\in\{0,\dots,n-1\}$ and $u\in\fconvsreg$. Then 
	$$
	[\Hess u(x)]^{lm}_{n-i}u_l(x) u_m(x)=|\nabla u(x)|^{n-i+1}\tau_{n-i-1}(u,x)
	$$
	for every $x\ne0$.
\end{lemma}

\begin{proposition}\label{prop int case n-2} 
	Let $i\in\{1,\dots,n-1\}$ and  $u\in\fconvsreg$. If $\gamma: (0,\infty)\to\R$ is differentiable, then 
	\begin{align*}
    \div_m\left(
	\gamma(|\nabla u(x)|)[\Hess u(x)]^{lm}_{n-i} u_l(x)
	\right)&=
	[\Hess u(x)]_{n-i}\Big((n-i)\gamma(|\nabla u(x)|)+\gamma'(|\nabla u(x)|)|\nabla u(x)|\Big)\\
	&\quad -\gamma'(|\nabla u(x)|)|\nabla u(x)|^{n-i+1}\tau_{n-i}(u,x)
	\end{align*} 
	for every $x\ne0$.
\end{proposition}
\begin{proof}
	We have
	\begin{eqnarray*}
		\div_m\left(
		\gamma(|\nabla u|)[\Hess u]^{lm}_{n-i}u_l
		\right)=
		\gamma(|\nabla u|)\div_m([\Hess u]^{lm}_{n-i}u_l)+
		\frac{\gamma'(|\nabla u|)} {|\nabla u|}
		[\Hess u]^{lm}_{n-i}  {u_k} u_m u_{km}.
	\end{eqnarray*} 
	Note that, by Lemma \ref{fc lemma 1} and \eqref{fc1},
	$$
	\div_m([\Hess u]^{lm}_{n-i}u_l)=[\Hess u]^{lm}_{n-i} u_{lm}=(n-i)[\Hess u]_{n-i}.
	$$
	The conclusion follows immediately from Lemma \ref{lemma BNST}.
\end{proof}

Recall that given $\zeta\in\Had{i}{n}$ for $1\leq i \leq n-1$, we associate two functions $\eta,\rho\in C_c([0,\infty))$ with $\zeta$, defined by
$$
\eta(t)=\int_t^{\infty}  s^{n-i-1}\zeta(s) \d s,\qquad\rho(t)=t^{n-i} \zeta(t) + (n-i)\eta(t)
$$
for $t\in[0,\infty)$.

\goodbreak
The main result of this part is the following proposition.

\begin{proposition}\label{prop:hessian_int_by_parts}
	Let $1\leq i\leq n-1$ and  $\zeta\in\Had{i}{n}$. For every $u\in\fconvsreg$ and $t_1, t_2$ with $0<t_1<t_2$,
	\begin{align*}
	\int_{\{t_1<u\le t_2\}} \zeta(|\nabla u(x)|) [\Hess u(x)]_{n-i} \d x &= \int_{\{t_1<u\le t_2\}} \rho(|\nabla u(x)|) \tau_{n-i}(u,x) \d x\\ 
	&\quad- \int_{\{u=t_2\}} \eta(|\nabla u(x)|) \tau_{n-i-1}(u,x) \d \hm^{n-1}(x)\\
	&\quad+\int_{\{u=t_1\}} \eta(|\nabla u(x)|) \tau_{n-i-1}(u,x) \d \hm^{n-1}(x).
	\end{align*}
\end{proposition}

\begin{proof} For $r>0,$ we set
	\begin{equation*}
	\gamma(r):=-\frac1{r^{n-i}}\int_r^{\infty}s^{n-i-1}\zeta(s) \d s.
	\end{equation*}
	Then
	$$
	\gamma'(r)= \frac{n-i}{r^{n-i+1}} \int_r^{\infty} s^{n-i-1}\zeta(s) \d s + \frac{\zeta(r)}{r}
	$$
	and
	$$
	(n-i)\gamma(r)+r\gamma'(r)=\zeta(r)
	$$
	for every $r>0$. Moreover,
	$$
	r^{n-i+1}\gamma'(r)=(n-i)\eta(r)+ r^{n-i}\zeta(r)=\rho(r)
	$$
	for every $r>0$.
	By the previous relations, Proposition \ref{prop int case n-2} and the divergence theorem, we have
	\begin{eqnarray*}
		\int_{\{t_1 <u\le t_2 \}}\zeta(|\nabla u|)[\Hess u]_{n-i}\d x
		&=&\int_{\{t_1 <u\le t_2 \}}\big((n-i)\gamma(|\nabla u|)+|\nabla u|\gamma'(|\nabla u|)\big)[\Hess u]_{n-i}\d x\\
		&=&\int_{\{t_1 <u\le t_2 \}}|\nabla u|^{n-i+1}\gamma'(|\nabla u|)\tau_{n-i}\d x\\
		&&+\int_{\{t_1 <u\le t_2 \}} \div_m\left(\gamma(|\nabla u|)[\Hess u]^{lm}_{n-i}u_l\right)\d x\\
		&=&\int_{\{t_1 <u\le t_2 \}} \rho(|\nabla u|)\tau_{n-i}\d x+\int_{\{u=t_2\}}\frac{\gamma(|\nabla u|)}{|\nabla u|}[\Hess u]^{lm}_{n-i}u_lu_m\d\hm^{n-1}\\
		&&-\int_{\{u=t_1\}}\frac{\gamma(|\nabla u|)}{|\nabla u|}[\Hess u]^{lm}_{n-i}u_lu_m\d\hm^{n-1}.
	\end{eqnarray*}
	In the last equation, we have used the fact that the outer unit normal to the set $\{t_1< u\le t_2 \}$ at a point $x\in\{u= t_2 \}$ is given by
	$$
	\frac{\nabla u(x)}{|\nabla u(x)|},
	$$
	while at a point $x\in\{u=t_1\}$ it is given by
	$$
	-\frac{\nabla u(x)}{|\nabla u(x)|}.
	$$
	The conclusion follows from Lemma \ref{lemma Reilly}. 
\end{proof}

\subsubsection{An Estimate}

Given a real-valued function $u$ defined in a subset $V$ of $\R^n$, the Lipschitz constant of $u$ in $U\subset V$ is defined as
$$
\Lip_U (u):=\sup\left\{\frac{|u(x)-u(y)|}{|x-y|}\colon x,\, y\in U,\, x\ne y\right\}.
$$
For $u\in\fconvs$ and $t>0$, we set 
$$\Lip(u,t):=\Lip_{\{u\le t\}}(u)$$
and note that for $u\in\fconvs\cap C^1(\R^n)$,
\begin{equation}\label{Liplim}
\lim_{t\to 0+} \Lip(u,t)=0
\end{equation}
if $0=u(0)\le u(x)$ for $x\in\R^n$.

\begin{lemma}
	\label{estimate smooth case}
	Let $1\leq i\leq n-1$ and  $\zeta\in\Had{i}{n}$.
	If $u\in\fconvsreg$, then
	$$\left\lvert \int_{\{t_1<u\le t_2\}} \zeta(|\nabla u(x)|) [\Hess u(x)]_{n-i} \d x\right\rvert \leq \alpha\,V_i(\{u\le t_2\}) \Big(\max_{[0,\Lip(u,t_2)]} |\rho| + \max_{[0,\Lip(u,t_2)]} |\eta|\Big)$$
	for every $0<t_1<t_2$ where $\alpha$ is a positive constant depending only on $n$ and $i$.
\end{lemma}
\begin{proof}
	By Lemma~\ref{le:tau_1} (and the monotonicity of intrinsic volumes)
	\begin{align*}
	\left\lvert\int_{\{t_1<u\le t_2\}} \rho(|\nabla u(x)|) \tau_{n-i}(u,x) \d x\right\rvert 
	&\leq \alpha \left( V_{i}(\{u\le t_2\}) - V_{i}(\{u\le t_1\}) \right) \max_{[0,\Lip(u, t_2)]} |\rho|\\
	&\leq \alpha\,  V_{i}(\{u\le t_2\}) \max_{[0,\Lip(u, t_2)]} |\rho|,
	\end{align*}
	where $\alpha>0$ only depends on $n$ and $i$. Similarly, by Lemma~\ref{le:tau_2} 
	$$\left\lvert\int_{\{u=t_j\}} \eta(|\nabla u(x)|) \tau_{n-i-1}(u,x) \d \hm^{n-1}(x)\right\rvert\leq \alpha\, V_i(\{u\le t_2\})\max_{[0,\Lip(u,t_2)]} |\eta|$$
	for $j\in\{1,2\}$. The conclusion now follows from Proposition~\ref{prop:hessian_int_by_parts}.
\end{proof}

\subsection{Proof of Theorem \ref{main one way}}
Throughout this section, fix $j\in\{1,\dots, n-1\}$ and $\zeta\in\Had{j}{n}$. We will associate the two functions $\rho,\eta\in C_c([0,\infty))$ defined in (\ref{eq:def_eta}) and (\ref{eq:def_rho})  with $\zeta$.

First, we prove that we may reduce to the case
$$
\eta(0)=\rho(0)=0.
$$
Indeed, let
$$
\eta_0:=\lim_{r\to0^+}\eta(r),
$$
and choose $\zeta_0\in C_c([0,\infty))$ such that
$$
\int_0^{\infty}s^{n-j-1}\zeta_0(s)\d s=\eta_0.
$$
Then $\bar\zeta\colon(0,\infty)\to\R$, defined by
$$
\bar\zeta(r):=\zeta(r)-\zeta_0(r),
$$
belongs to $\Had{j}{n}$. The function $\bar\eta$, defined by
$$
\bar\eta(r):=\int_r^{\infty}s^{n-j-1}\bar\zeta(s)\d s,
$$
verifies $\bar\eta(0)=0$.
Moreover, as $\zeta_0\in C_c([0,\infty))$, the functional $\oZ_{\zeta_0}\colon\fconvs\to\R$, given by
$$
\oZ_{\zeta_0}(u):=\int_{\R^n\times \R^n}\zeta_0(|y|)\d\Theta_j^n(u,y),
$$
defines, by Lemma \ref{fine zeta}, a continuous, epi-translation and rotation invariant valuation on $\fconvs$. We have then proved that $\zeta$ can be written as 
$$
\zeta=\bar\zeta+\zeta_0,
$$
where $\zeta_0$ gives rise to a continuous valuation on $\fconvs$, while $\bar\zeta\in\Had{j}{n}$, and the function $\bar\eta$ associated to $\bar\zeta$ verifies $\bar\eta(0)=0$. Finally, note that it follows from the relation
$$
\bar\rho(r)=(n-j)\bar\eta(r)+r^{n-j}\bar\zeta(r),
$$
that $\bar\rho(0)=0$.

In the following, let 
$$\fconvso := \{u\in\fconvs\colon u(0)=0\leq u(x) \text{ for } x\in\R^n\}$$
and 
$$T_j(u):=\{t>0: \Theta_j^n(u,\{u=t\}\times\R^n)=0\}.$$
Since $\Theta_j^n(u,\cdot)$ is a locally finite measure, the complement of $T_j(u)$ in $(0,\infty)$ has at most countably many points.

\begin{lemma}
	\label{le:sequence_v_int}
	Let $u\in\fconvso\cap C^1(\R^n)$. Then there exists a sequence $u_k$ of functions from $\fconvsreg$ such that $u_k$ epi-converges to $u$ and 
	$$
	\lim_{k\to\infty} \int_{\{t_1 < u_k\le  t_2\}} \zeta(|\nabla u_k(x)|) [\Hess u_k(x)]_{n-j} \d x = \int_{ \{(x,y) \colon t_1 < u(x) \leq t_2\}} \zeta(|y|) \d\Theta_{j}^n(u,(x,y))
	$$
	for all $0<t_1<t_2$ with $t_1,t_2\in T_j(u)$.
\end{lemma}

\begin{proof} The proof is divided into several steps.\smallskip
	
	\noindent{(i)} There exists a sequence $u_k$ of functions from $\fconvsreg$  that epi-converges to $u$. Indeed, by a standard mollification procedure, we first find a sequence of convex functions $\bar u_k$ of class $C^2(\R^n)$ such that $0=\bar u_k(0)\le \bar u_k(x)$ for every $x\in\R^n$ and $k\in\N$, and $\bar u_k$ converges to $u$ on compact sets (in particular, $\bar u_k$ epi-converges to $u$). Then we define, for every $x\in\R^n$ and $k\in\N$,
	$$
	u_k(x):=\bar u_k(x)+\frac1k |x|^2.
	$$
	Then clearly $u_k\in\fconvso\cap C^2(\R^n)$, and $u_k$ epi-converges to $u$. Moreover, $\Hess u_k(x)>0$ for every $x\ne0$ and $k\in\N$, and this implies that the boundary of $\{u_k\le t\}$ is of class $C^{2}$ with positive curvature for every $t>0$. Moreover, as $u\in C^1(\R^n)$, the sequence $\nabla u_k$ converges to $\nabla u$ uniformly on compact sets.
	
	\medskip
	\noindent{(ii)} For $t_1, t_2\in T_j(u)$ such that $0<t_1<t_2$,
	\begin{equation}\label{limit}
	\lim_{k\to\infty} \int_{ \{(x,y) \colon t_1 < u(x) \leq t_2\}}
	\zeta(|y|) \d\Theta_{j}^n(u_k,(x,y))
	= \int_{ \{(x,y) \colon t_1 < u(x) \leq t_2\}} \zeta(|y|) \d\Theta_{j}^n(u,(x,y)).
	\end{equation}
    For $\bar t>0$ fixed, there exists $\alpha>0$ such that 
    $$
    |\nabla u(x)\vert\ge2\alpha\qquad\mbox{ for every $x$ such that $u(x)>\bar t$}.
    $$ 
    Since the sequence $u_k$ converges to $u$ in the $C^1$-norm on compact sets, we may assume that 
    $$
    |\nabla u_k(x)|\ge \alpha\qquad\mbox{ for every $x$ such that $u_k(x)>\bar t$}
    $$

    \smallskip\noindent
    for $k$ sufficiently large (independent of $x$). Let $\bar\zeta$ be a continuous function in $[0,\infty)$, which coincides with $\zeta$ in $[\alpha,\infty)$ and hence has compact support. We have
	$$
	\int_{ \{(x,y) \colon t_1 < u(x) \leq t_2\}}
	\zeta(|y|) \d\Theta_{j}^n(u_k,(x,y))=
	\int_{\{(x,y)\colon t_1 < u(x) \leq t_2\}}
	\bar\zeta(|y|) \d\Theta_{j}^n(u_k,(x,y)),
	$$
	for every $\bar t < t_1 < t_2$  and for every $k$ sufficiently large, and the corresponding statement holds for $u_k$ replaced by $u$.
    As $u_k$ epi-converges to $u$ and as $t_1,t_2\in T_j(u)$, it follows from  (\ref{weak_cont}) that 
	$$
	\lim_{k\to\infty}\int_{ \{t_1 < u \leq t_2\}\times \R^n}
	\zeta(|y|) \d\Theta_{j}^n(u_k,(x,y))=
	\int_{ \{t_1 < u \leq t_2\}\times \R^n}
	\zeta(|y|) \d\Theta_{j}^n(u,(x,y)).
	$$
   	This proves that \eqref{limit} holds for all $t_1, t_2 \in T_j(u)$ such that $t_2>t_1$ and $t_1,t_2>\bar t$. As $\bar t>0$ was arbitrary, the claim is proved.
	
	\medskip
	\noindent {(iii)} For every $t>0$,
	\begin{equation}\label{limit 2}
	\lim_{k\to\infty}\Big|\int_{\{ u\ge t\}} \zeta(|\nabla u_k(x)|)[\Hess u_k(x)]_{n-j}\d x
	-\int_{\{u_k\ge t\}}\zeta(|\nabla u_k(x)|)[\Hess u_k(x)]_{n-j}\d x
	\Big|=0.
	\end{equation}
	We first note that in the previous relation, we may replace $\zeta$ by $\bar\zeta$ (where $\bar\zeta$ is chosen as in the previous step). Let $\bar\zeta_+$ and $\bar\zeta_-$ denote the positive and negative parts of $\bar\zeta$; these are continuous functions in $[0,\infty)$, with compact support.  
	As $\bar\zeta=\bar\zeta_+-\bar\zeta_-$, it is enough to prove \eqref{limit 2} for $\bar\zeta_+$ and $\bar\zeta_-$ separately. In other words, we may reduce to prove \eqref{limit 2} under the assumption that $\zeta$ is non-negative and belongs to $C_c([0,\infty))$. 
	
	For every $k\in\N$, let 
	$$
	t_{k,1}=\max\{r>0\colon \{u_k\le r\}\subset \{u\le t \} \},
	$$
	and 
	$$
	t_{k,2}=\min\{r>0\colon \{u_k\le r\}\supset \{u\le t \} \}.
	$$
	Equivalently, $\{u_k \le t_{k,1}\}$ is the largest sublevel set of $u_k$ contained in $\{u\le t \}$, while $\{u_k \le t_{k,2}\}$ is the smallest sublevel set of $u_k$ containing $\{u\le t \}$. Clearly
	$$
	\{u_k \le t_{k,1}\}\subset \{u\le t \}\subset \{u_k \le t_{k,2}\}
	$$
	for every $k\in\N$. Moreover, by uniform convergence, we have
	$$
	\lim_{k\to\infty}\{u_k \le t_{k,1}\}=\lim_{k\to\infty}\{u_k \le t_{k,2}\}=\{u\le t \}
	$$
	(here, convergence is in the Hausdorff metric). As $\zeta$ is non-negative,
	\begin{align*}
	\int_{\R^n\setminus\{u_k \le t_{k,2}\}}\zeta(|\nabla u_k(x)|)[\Hess u_k(x)]_{n-j}\d x &\le \int_{\R^n\setminus \{u \le t\}}\zeta(|\nabla u_k(x)|)[\Hess u_k(x)]_{n-j}\d x\\
	&\le \int_{\R^n\setminus \{u_k \le t_{k,1}\}}\zeta(|\nabla u_k(x)|)[\Hess u_k(x)]_{n-j}\d x.
	\end{align*}
	Therefore, the difference
	\begin{equation}\label{difference}
	\alpha_k:=\int_{\{ u\ge t\}}\zeta(|\nabla u_k(x)|)[\Hess u_k(x)]_{n-j}\d x
	-\int_{\{ u_k\ge t\}}\zeta(|\nabla u_k(x)|)[\Hess u_k(x)]_{n-j}\d x\nonumber
	\end{equation}
	is bounded from below by
	\begin{equation}\label{lower bound}
	\int_{\{ a_k\le u_k\le b_k \}}\zeta(|\nabla u_k(x)|)[\Hess u_k(x)]_{n-j}\d x,
	\end{equation}
	where
	$a_k=\min\{t,t_{k,2}\}$ and $b_k=\max\{t,t_{k,2}\}$.
	Similarly, $\alpha_k$ is bounded from above by
	\begin{equation}\label{upper bound}
	\int_{\{ c_k\le u_k\le d_k \}}\zeta(|\nabla u_k(x)|)[\Hess u_k(x)]_{n-j}\d x,
	\end{equation}
	where $c_k=\min\{t,t_{k,1}\}$ and $d_k=\max\{t,t_{k,1}\}$. Note that the sequences $a_k$, $b_k$, $c_k$ and $d_k$ converge to $t$ as $k\to \infty$. We prove that the integral in \eqref{lower bound} converges to zero as $k\to\infty$. The same, with a similar proof, holds for \eqref{upper bound}. This will complete the proof of this step.
	
	By Proposition \ref{prop:hessian_int_by_parts}, we have
	\begin{align*}
	\int_{\{ a_k\le u_k\le b_k \}}\zeta(|\nabla u_k(x)|)[\Hess u_k(x)]_{n-j}\d x&=
	\int_{\{ a_k\le u_k\le b_k \}}\rho(|\nabla u_k(x)|)\tau_{n-j}(u_k,x)\d x\\
	&\quad-\int_{\{u_k=b_k\}}\eta(|\nabla u_k(x)|)\tau_{n-j-1}(u_k,x)\d\hm^{n-1}(x)\\
	&\quad+\int_{\{ u_k=a_k\}}\eta(|\nabla u_k(x)|)\tau_{n-j-1}(u_k,x)\d\hm^{n-1}(x),
	\end{align*}
	where the functions $\rho$ and $\eta$ are associated to $\zeta$ as in Proposition \ref{prop:hessian_int_by_parts}. For the first term on the right side, we may write, using the continuity of $\rho$ and Lemma \ref{le:tau_1},
	$$
	\left|\int_{\{ a_k\le u_k\le b_k \}}\rho(|\nabla u_k(x)|)\tau_{n-j}(u,x)\d x\right|\le
	\alpha\left(
	V_j(\{u_k\le b_k\})-V_j(\{u_k\le a_k\})
	\right),
	$$
	where $\alpha>0$ depends on $n$, $j$ and $\rho$. By the uniform convergence and the continuity of intrinsic volumes, the right side of the previous inequality tends to 0 as $k\to \infty$. 
	
	Next, we consider  
	\begin{equation*}
	\int_{\{u_k=b_k\}}\eta(|\nabla u_k(x)|)\tau_{n-j-1}(u_k,x)\d\hm^{n-1}(x).
	\end{equation*}
	This integral can be written as
	$$
\binom{n-1}{j} \int_{\R^n}\eta(|\nabla u_k(x)|)\d C_{j}(\{u_k\le b_k\},x),
	$$
	where $C_{j}(\{u_k\le b_k\} ,\cdot)$ is the $j$th curvature measure of $\{u_k\le b_k\}$ (see \cite[Chapter 4]{Schneider:CB2} for the definition of curvature measures). Note that by the convergence of $\{u_k \le b_k\}$ to $\{u\le t \}$ and \cite[Theorem 4.2.1]{Schneider:CB2},
	the curvature measures $C_{j}(\{u_k \le b_k\},\cdot)$ converge weakly to $C_{j}(\{u\le t \},\cdot)$. Moreover, the support of the curvature measures of a convex body is contained in the boundary of the convex body. Hence there exists a compact set $C\subset\R^n$, such that the supports of $C_j(\{u_k\le b_k\},\cdot)$ and of $C_j(\{u\le t\}, \cdot)$ are contained in $C$ for every $k$.   
	Taking the continuity of $\eta$ and the uniform convergence of $\nabla u_k$ to $\nabla u$ on compact sets into account, we deduce that 
	\begin{align*}
	\lim_{k\to\infty}\int_{\{u_k=b_k\}}\eta(|\nabla u_k(x)|)\tau_{n-j-1}(u_k,x)\d\hm^{n-1}(x)=\binom{n-1}{j}\int_{\R^n}\eta(|\nabla u(x)|)\d C_{j}(\{u\le t\},x).
	\end{align*}
	By a similar argument, we may prove that
	\begin{eqnarray*}
		\lim_{k\to\infty}\int_{\{u_k =a_k\}}\eta(|\nabla u_k(x)|)\tau_{n-j-1}(u_k,x)\d\hm^{n-1}(x)
		=\binom{n-1}{j}\int_{\R^n}\eta(|\nabla u(x)|)\d C_{j}(\{u\le t\},x).
	\end{eqnarray*}
	Thus, the integral in \eqref{lower bound} converges to zero as $k\to\infty$, which concludes the proof of this step.
	
	\medskip
	\noindent {(iv)} By the previous step we have for every $t_1,t_2\in T_j(u)$ such that $0<t_1<t_2$,
	\begin{equation*}
	\lim_{k\to\infty} \left|\int_{\{ t_1< u\le t_2\}}\zeta(|\nabla u_k(x)|)[\Hess u_k(x)]_{n-j}\d x
	-\int_{\{t_1<u_k\le t_2\}}\zeta(|\nabla u_k(x)|)[\Hess u_k(x)]_{n-j}\d x\right|=0.
	\end{equation*}
	This, together with \eqref{limit}, concludes the proof of this lemma.
\end{proof}

\begin{lemma}\label{lemma estimate}
	Let $u\in\fconvso\cap C^1(\R^n)$. Then for $t_1,t_2\in T_j(u)$ such that $0<t_1<t_2$, we have 
	\begin{equation*}\label{estimate}
	\Big\lvert\int_{\ \{(x,y)\colon t_1 < u(x)\leq t_2\}} \zeta(|y|) \d \Theta_j^n(u,(x,y))\Big\rvert \leq \alpha\,V_j(\{u\le t_2\}) \Big(\max_{[0,\Lip(u,t_2)]} |\rho| + \max_{[0,\Lip(u,t_2)]}|\eta|\Big),
	\end{equation*}
	where $\alpha$ is a positive constant only depending on $n$ and $j$.
\end{lemma}

\begin{proof} 
		The validity of the claimed inequality follows from Lemma~\ref{estimate smooth case}, Lemma \ref{le:sequence_v_int}, and the convergence of Lipschitz constants and of sublevel sets under uniform convergence.
\end{proof}

\begin{lemma}
\label{lem:lim_int_sup_lvl_set}
	For every $u\in\fconvso\cap C^1(\R^n)$, the limit
	\begin{equation*}
	\lim_{\begin{subarray}{c}\\ t\to0\\[1pt]t\in T_j(u)\end{subarray}}\,\,\, \int_{\begin{subarray}{c}\\[10pt] \{(x,y) \colon u(x)> t\}\end{subarray}}\!\! \zeta(|y|) \d \Theta_j^n(u,(x,y))
	\end{equation*}
	exists and is finite.
\end{lemma}

\goodbreak

\begin{proof} From Lemma \ref{lemma estimate}, we obtain that for every $t_1, t_2\in T_j(u)$ such that $0<t_1<t_2$, 
	\begin{align*}
	\Big\lvert\int_{ \{(x,y)\colon t_1 < u(x)\leq t_2\}} \zeta(|y|) & \d \Theta_j^n(u,(x,y))\Big\rvert\leq \alpha \,V_j(\{u\le t_2\}) \Big(\max_{[0,\Lip(u,t_2)]} |\rho| + \max_{[0,\Lip( u,t_2)]} |\eta |\Big).
	\end{align*}
	The conclusion now follows from (\ref{Liplim}) and the fact  that both $\eta$ and $\rho$ are continuous on $[0,\infty)$ and  vanish at $t=0$.
\end{proof}

Lemma \ref{lem:lim_int_sup_lvl_set} and the fact that $\my_\lambda u\in\fconvs\cap C^1(\R^n)$ allow us to make the following definition. For $u\in\fconvso$ and $\lambda>0$, we set
\begin{equation}\label{my definition}
\ZZZ{j}{\zeta}{\lambda}(u) := \lim_{\begin{subarray}{c}\\ {t\to0}\\[1pt] {t\in T_j(\my_\lambda u)}\end{subarray}}\,\,\, 
\int_{\!\!\begin{subarray}{c}\\[10pt] \{(x,y) \colon \my_\lambda u(x)> t\}\end{subarray}}\!\!\! \zeta(|y|) \d \Theta_j^n(\my_\lambda u,(x,y)).
\end{equation}
An immediate consequence of this definition and Lemma~\ref{lemma estimate} is the following result.

\begin{lemma}
	\label{le:estimate_int_u_lambda_leq_t}
	Let $u\in\fconvso$ and $\lambda>0$. If $t\in T_j(\my_\lambda u)$, then
\begin{multline*}
	\Big\lvert\ZZZ{j}{\zeta}{\lambda}(u)-\int\limits_{\{(x,y) \colon \my_\lambda u(x)>t\}} \zeta(|y|) \d \Theta_j^n(\my_\lambda u,(x,y))\Big\rvert \\
	\leq \alpha\, V_j(\{\my_\lambda u\le t\}) \big(\max_{[0,\Lip(\my_\lambda u, t)]} |\rho| + \max_{[0,\Lip(\my_\lambda u, t)]} |\eta| \big),
\end{multline*}
	where $\alpha$ is a positive constant only depending on $n$ and $j$.
\end{lemma}

We introduce the following auxiliary functions.  For $r>0$, we define $\zeta_r\colon[0,\infty)\to\R$ as 
	\begin{equation*}
	\zeta_r(t):=
	\left\{
	\begin{array}{lll}
	\zeta(t)\quad&\mbox{for $t\ge r$,}\\[6pt]
	\zeta(r)\quad&\mbox{for $0\le t<r$.}
	\end{array}
	\right.
	\end{equation*}
	Note that $\zeta_r\in C_c([0,\infty))$. We also introduce the corresponding functions $\eta_r, \rho_r\colon[0,\infty)\to\R$, defined as
	$$
	\eta_r(t):=\int_t^{\infty}s^{n-j+1}\zeta_r(s)\d s,\qquad
	\rho_r(t):=t^{n-1}\zeta_r(t)+(n-j)\eta_r(t).
	$$
	Clearly, $\eta_r,\rho_r\in C_c([0,\infty))$. For $t>r$, we have 
	$
	\eta_r(t)=\eta(t),
	$
	and, for $0\le t<r$, 
	\begin{equation*}
		\eta_r(t)=\int_t^{r}s^{n-j-1}\zeta(r)\d s+\int_r^{\infty}s^{n-j-1}\zeta(s)\d s
		=\zeta(r)\,\frac{r^{n-j}-t^{n-j}}{n-j}+\eta(r).
	\end{equation*}
	Hence
	\begin{equation*}
	\eta_r(t)=
	\left\{
	\begin{array}{lll}
	\eta(t)\quad&\mbox{for $t\ge r$,}\\[6pt]
	\zeta(r)\,\dfrac{r^{n-j}-t^{n-j}}{n-j}+\eta(r)\quad&\mbox{for $0\le t<r$.}
	\end{array}
	\right.
	\end{equation*}
	Concerning the function $\rho_r$, we have $\rho_r(t)= \rho(t)$ for $t\ge r$ and
	\begin{equation*}
		\rho_r(t)=t^{n-j}\zeta(r)+(n-j)\zeta(r)\,\frac{r^{n-j}-t^{n-j}}{n-j}+(n-j)\eta(r)
		=r^{n-j}\zeta(r)+(n-j)\eta(r)=\rho(r)
	\end{equation*}
	for $0\le t < r$.
	Hence
	\begin{equation*}
	\rho_r(t)=
	\left\{
	\begin{array}{lll}
	\rho(t)\quad&\mbox{for $t\ge r$,}\\[6pt]
	\rho(r)\quad&\mbox{for $0\le t<r$.}
	\end{array}
	\right.
	\end{equation*}
For every $\delta>0$, we obtain
	\begin{eqnarray}\label{estimates!}
		\max_{[0,\delta]}|\eta_r|&\le&|\eta(r)|+\frac2{n-j}\,|r^{n-j}\zeta(r)|+\max_{[0,\delta]}|\eta|,\nonumber\\[-6pt]
&&\\[-6pt]
		\max_{[0,\delta]}|\rho_r|&\le&|\rho(r)|+\max_{[0,\delta]}|\rho|. \nonumber
	\end{eqnarray}
These auxiliary functions are used in the following two lemmas.

\goodbreak

\begin{lemma}\label{continuity lemma}
	If $u_k\in\fconvso$ epi-converges to  $u\in\fconvso$, then
	$$
	\lim_{k\to\infty} \ZZZ{j}{\zeta}{\lambda}(u_k) = \ZZZ{j}{\zeta}{\lambda}(u)
	$$
	for every $\lambda>0$.
\end{lemma}
\begin{proof} 
	Fix $\lambda>0$ and set $\bar u=\my_\lambda u$ and $\bar u_k=\my_\lambda u_k$. For any fixed $\bar t>0$, there exists a constant $\alpha>0$ (depending on $u$ and $\bar t$) such that
	$$|\nabla \bar u(x)|>2\alpha\qquad\mbox{for every $x$ such that $\bar u(x)\ge \bar t$}.
    $$
	As $u_k$ epi-converges to $u$, the sequence $\bar u_k$ converges to $\bar u$ and the sequence $\nabla\bar u_k$ to $\nabla\bar u$. In both cases, the convergence is uniform on compact sets. Therefore, we have 
	$$
	|\nabla \bar u_k(x)|>\alpha\qquad\mbox{for every $x$ such that $\bar u_k(x)\ge \bar t$}
    $$
    for $k$ sufficiently large.\goodbreak
    
	Set $r=\alpha$. We have
	\begin{eqnarray*}
		\left|\ZZZ{j}{\zeta}{\lambda}(u)-\ZZZ{j}{\zeta}{\lambda}(u_k)
		\right|
		\le \beta_0+\dots+\beta_4,
	\end{eqnarray*}
	where
	\begin{eqnarray*}
		\beta_0&=&\Big|\ZZZ{j}{\zeta_r}{\lambda}(u)-
		\ZZZ{j}{\zeta_r}{\lambda}(u_k)\Big|,\\
		\beta_1&=&\Big|\ZZZ{j}{\zeta}{\lambda}(u) -\int_{\{(x,y)\colon \bar u(x)> \bar t\}}\zeta(|y|)\d\Theta_j^n(\bar u,(x,y))\Big|,\\
		\beta_2&=&\Big|\ZZZ{j}{\zeta_r}{\lambda}(u) -\int_{\{(x,y)\colon \bar u(x)>\bar t\}}\zeta_r(|y|)\d\Theta_j^n(\bar u,(x,y))\Big|,\\
		\beta_3&=&\Big|\ZZZ{j}{\zeta}{\lambda}(u_k) -\int_{\{(x,y)\colon \bar u_k(x)> \bar t\}}\zeta(|y|)\d\Theta_j^n(\bar u_k,(x,y))\Big|,\\
		\beta_4&=&\Big|\ZZZ{j}{\zeta_r}{\lambda}(u_k) -\int_{\{(x,y)\colon \bar u_k(x)> \bar t\}}\zeta_r(|y|)\d\Theta_j^n(\bar u_k,(x,y))\Big|.
	\end{eqnarray*}
	These quantities depend, in general, on $\bar t$ and $k$ (note that $r$ depends on $\bar t$). We fix $\varepsilon>0$  and first consider the terms $\beta_1$ and $\beta_3$. By Lemma \ref{le:estimate_int_u_lambda_leq_t}, the continuity of $\eta$ and $\rho$ at $t=0$, the relations $\eta(0)=\rho(0)=0$, and the uniform convergence of $\bar u_k$ to $\bar u$ on compact sets, we may choose $\bar t_1>0$ so that
	$
	\beta_1,\; \beta_3\le\varepsilon
	$
	for every 
	$$\bar t\in(0,\bar t_1)\cap T_j(\bar u)\cap \bigcap\nolimits_{k\ge k_1}T_j(\bar u_k)$$ 
	with $k_1$ sufficiently large. Note that this set has full measure in $(0,\bar t_1)$.

Next, we deal with the term $\beta_2$. For $t>0$, set
$$
r(t):=\inf\big\{\vert \nabla \bar u(x)\vert: x\  \mbox{ such that}\ \bar  u(x)>t\big\}.
$$
Note that $r(t)$ tends to zero as $t\to0$. By Lemma~\ref{le:estimate_int_u_lambda_leq_t}, we have for $\bar t\in T_j(\bar u)$,
$$
\beta_2\le \alpha\, V_j(\{\bar u\le \bar t\}) \big(\max_{[0,\Lip(\bar u ,\bar t)]} |\rho_{r(\bar t)}| + \max_{[0,\Lip(\bar u ,\bar t)]} |\eta_{r(\bar t)}| \big).
$$
By \eqref{Liplim} and \eqref{estimates!} combined with the conditions on $\rho$ and $\eta$, we deduce that there exists $\bar t_2>0$ so that
	$
	\beta_2\le\varepsilon
	$
for every $\bar t\in(0,\bar t_2)\cap T_j(\bar u)$.  

\goodbreak
We proceed in a similar way for $\beta_4$. For $k\in\N$ and $t>0$, let
$$
r_k(t):=\inf\big\{\vert \nabla\bar u_k(x)\vert: x\  \mbox{ such that}\ \bar  u_k(x)>t\big\}.
$$
As before, $r_k(t)$ tends to zero as $t\to0$. Moreover, for $\bar t\in T_j(\bar u_k)$,
$$
\beta_4\le \alpha\, V_j(\{\bar u_k\le \bar t\}) \big(\max_{[0,\Lip(\bar u_k, \bar t]} |\rho_{r_k(\bar t)}| + \max_{[0,\Lip(\bar u_k,\bar t)]} |\eta_{r_k(\bar t)}| \big).
$$
By \eqref{Liplim} and \eqref{estimates!} combined with the uniform convergence of $\bar u_k$ to $\bar u$ on compact sets, there exists $\bar t_3>0$ such that
$
\beta_4\le\varepsilon
$
for every $\bar t\in(0,\bar t_3)\cap \bigcap_{k\ge k_3} T_j(\bar u_k)$ with $k_3$ sufficiently large. 

Finally, for every fixed $r>0$, and hence for every fixed $\bar t >0$, the term $\beta_0\to0$ as $k\to\infty$, since $\zeta_r\in C_c([0,\infty))$ and by Lemma \ref{le:hessian_val_zeta_cont_comp_supp}. This concludes the proof.
\end{proof}

\begin{lemma}\label{lemma for the valuation property} 
For every $u\in\fconvso$ and $\lambda>0$,
$$
\lim_{r\to0^+}\ZZZ{j}{\zeta_r}{\lambda}(u)=\ZZZ{j}{\zeta}{\lambda}(u).
$$
\end{lemma}

\begin{proof}
For $r>0$, let
$$
t(r):=\inf\big\{
t>0: \vert \nabla\my_\lambda u(x)\vert\ge r\ 
\mbox{ for all $x$ such that  $\my_\lambda u(x)>t$}
\big\}.
$$
This defines a monotone function of $r$, and
$$
\lim_{r\to0} t(r)=0
$$
since $\my_\lambda u$ is of class $C^1$.

For every $r>0$, let $\bar t(r)\in T_j(\my_\lambda u)$ be such that $t(r)\le \bar t(r)\le t(r)+r$. We have
\begin{eqnarray*}
\Big\vert\ZZZ{j}{\zeta}{\lambda}(u)-\ZZZ{j}{\zeta_r}{\lambda}(u)\Big\vert&\le& \Big\vert\ZZZ{j}{\zeta}{\lambda}(u)-\int_{\{(x,y)\colon\my_\lambda u(x)> \bar t(r)\}}\zeta(\vert y\vert)\d\Theta_j^n(\my_\lambda u,(x,y))\Big\vert\\
&&+\Big\vert\ZZZ{j}{\zeta_r}{\lambda}(u)-\int_{\{(x,y)\colon\my_\lambda u(x)> \bar t(r)\}}\zeta_r(\vert y\vert)\d\Theta_j^n(\my_\lambda u,(x,y))\Big\vert.
\end{eqnarray*}
Let $r\to0^+$. The conclusion follows from Lemma \ref{le:estimate_int_u_lambda_leq_t}, (\ref{Liplim}), and \eqref{estimates!}.
\end{proof}\goodbreak

We extend $\ZZZ{j}{\zeta}{\lambda}$ from  a functional defined on $\fconvso$  to a functional defined on $\fconvs$ in the following way. For every $u\in\fconvs$, there exists $u_0\in\fconvso$ such that $\epi(u_0)$ is a translate of $\epi (u)$. Indeed, if $x_0\in\R^n$ is such that $u$ attains its absolute minimum at $x_0$, it is sufficient to define $u_0$ for $x\in\R^n$ as
$$ 
u_0(x):=u(x+x_0)-\min\nolimits_{\R^n}u.
$$
We note that $u_0$ is not uniquely determined (as $x_0$ is not), but any two functions of this form can be obtained from each other by a translation of the epi-graph. Therefore,
$$
\ZZZ{j}{\zeta}{\lambda}(u):=\ZZZ{j}{\zeta}{\lambda}(u_0)
$$
is independent of the choice of the particular $u_0$. 
Clearly, $\ZZZ{j}{\zeta}{\lambda}$ extended in this way is epi-translation invariant and, as $\ZZZ{j}{\zeta}{\lambda}$ is rotation invariant on $\fconvso$, the same holds for its extension to $\fconvs$.
By Lemma \ref{continuity lemma}, this extension is also continuous.

For every $r>0$, we have $\zeta_r\in C_c([0,\infty))$. Hence, it follows from Lemma \ref{lemma for the valuation property}, that the functional $\ZZZ{j}{\zeta_r}{\lambda}$, defined in (\ref{my definition}), is
$$\ZZZ{j}{\zeta_r}{\lambda}(u) = \int_{\R^n\times\R^n} \zeta_r(|y|) \d \Theta_j^n(\my_\lambda u,(x,y))$$
for $u\in\fconvso$, extends by this representation to $\fconvs$ and defines there an  epi-translation invariant valuation. Hence, using the epi-translation invariance of both $\ZZZ{j}{\zeta}{\lambda}$ and $\ZZZ{j}{\zeta_r}{\lambda}$ on $\fconvs$,  we deduce from Lemma \ref{lemma for the valuation property} that also $\ZZZ{j}{\zeta}{\lambda}$ is a valuation on $\fconvs$. We conclude that $\ZZZ{j}{\zeta}{\lambda}$ is a continuous, epi-translation invariant and rotation invariant valuation on $\fconvs$.

The inclusion $\Had{j}{n}\subset\Had{i}{n}$ for $i\in\{1,\dots,j\}$ ensures that we may repeat this construction replacing $j$ by any $i\in\{1,\dots,j\}$. 
\begin{lemma}\label{my_cerv}
For every $i\in\{1,\dots,j\}$ and  $\lambda>0$, the function $\ZZZ{i}{\zeta}{\lambda}$ is a continuous, epi-translation and rotation invariant valuation on $\fconvs$.
\end{lemma}

For $u\in\fconvs \cap C^2_+(\R^n)$, Proposition \ref{existence in the smooth case} shows that
$$
\oZZ{j}{\zeta}(u):=\int_{\R^n} \zeta(|\nabla u(x)|) [\Hess u(x)]_{n-j} \d x
$$
is well-defined. Moreover, as $\Had{j}{n}\subset\Had{i}{n}$ for $i\in\{1,\dots,j\}$, we may also set
$$
\oZZ{i}{\zeta}(u):=\int_{\R^n} \zeta(|\nabla u(x)|) [\Hess u(x)]_{n-i} \d x
$$
for every $i=1,\dots,j$ and $u\in\fconvs \cap C^2_+(\R^n)$.
Since $\my_\lambda u\in \fconvs \cap C^2_+(\R^n)$ for such $u$  and $\lambda >0$, we have 
\begin{equation}\label{tlambda}
\oZZ{i}{\zeta}(\my_\lambda u)=\ZZZ{i}{\zeta}{\lambda}(u)
\end{equation}
by Proposition \ref{existence in the smooth case}.

\goodbreak
For the final step of the proof of Theorem \ref{main one way}, we require the following statement.

\begin{lemma}\label{vandermonde}
	There exist $\alpha_1,\ldots,\alpha_{j+1}\in\R$ such that 
	\begin{equation}\label{expansion}
	\oZZ{j}{\zeta}(u)= \sum_{i=1}^{j+1} \alpha_i \ZZZ{j}{\zeta}{i}(u)
	\end{equation}
for every $u\in\fconvs \cap C^2_+(\R^n)$.
\end{lemma}

\begin{proof} Let $u\in\fconvs \cap C^2_+(\R^n)$ and let $\lambda>0$. As usual, we denote by $u^*$ the conjugate of $u$. Note that
	$$
	(\my_\lambda u)^*(y)=u^*(y)+\lambda\frac{|y|^2}2
	$$
	and that
	$$
	[\Hess (\my_\lambda u)^*(y)]_j=[\Hess(u^*(y)+\lambda \frac{|y|^2}2)]_j=\sum_{i=0}^j \binom{n-i}{j-i} \lambda^{j-i}[\Hess u^*(y)]_i
	$$
	for every $y\in\R^n$. 
		Hence, it follows from (\ref{tlambda}) and (\ref{substitution}) that
	\begin{eqnarray*}
		\ZZZ{j}{\zeta}{\lambda}(u) &=& \int_{\R^n} \zeta(|\nabla (\my_\lambda u)(x)|) [\Hess (\my_\lambda u)(x)]_{n-j} \d x\\
		&=& \int_{\R^n} \zeta(|y|) [\Hess (\my_\lambda u)^*(y)]_{j} \d y\\
		&=& \sum_{i=0}^{j} \binom{n-i}{j-i} \lambda^{j-i} \int_{\R^n} \zeta(|y|) [\Hess u^*(y)]_{i} \d x\\
		&=& \sum_{i=0}^{j} \binom{n-i}{j-i} \lambda^{j-i} \int_{\R^n} \zeta(|\nabla u(x)|) [\Hess u(x)]_{n-i} \d x\\
		&=& \sum_{i=0}^{j} \binom{n-i}{j-i} \lambda^{j-i} \oZZ{i}{\zeta}(u).
	\end{eqnarray*}
Inverting the system that we obtain by writing the previous equation for $\lambda=1,\dots,n+1$, we deduce \eqref{expansion}. We use that the matrix is a Vandermonde matrix and, therefore, invertible.
\end{proof}

\goodbreak
Combined with Lemma \ref{my_cerv}, Lemma \ref{vandermonde} implies that $\oZZ{j}{\zeta}$ continuously extends to $\fconvs$. This completes the proof of Theorem \ref{main one way}.

\subsection{Representation formulas}

As an application of Theorem \ref{dual main one way}, we derive representation formulas for $\oZZ{j}{\zeta}$ and $\oZZ{j}{\zeta}^*$ on convex functions with certain regularity properties.

\begin{lemma}\label{preparatory lemma 4}
	Let $j\in\{0,\dots,n\}$ and $\zeta\in \Had{j}{n}$. We have 
	\begin{equation*}\label{identity 3}
	\oZZ{j}{\zeta}^*(v)=\int_{\R^n}\zeta(\vert x\vert)\d\Phi^n_{j}(v,x)
	\end{equation*}
	for $v\in\fconvfs$.
\end{lemma}
\begin{proof} By Theorem \ref{dual main one way} and (\ref{Hessian measures smooth case}), we know that for every $v\in\fconvf\cap C^2_+(\R^n)$, 
\begin{equation*}
\oZZ{j}{\zeta}^*(v)=\int_{\R^n}\zeta(\vert x\vert)\d\Phi^n_{j}(v,x)
=\int_{\R^n}\zeta(\vert x\vert)[\Hess v(x)]_j\d x.
\end{equation*}
Now, let $v\in\fconvfs$. 
There exists a sequence $v_k$ of functions from $\fconvf\cap C^2_+(\R^n)$  that converges to $v$. Indeed, by a standard mollification procedure, we first find a sequence of convex functions $\bar v_k$ of class $C^2(\R^n)$ that converges to $v$ on compact sets. Then we define, for every $x\in\R^n$ and $k\in\N$,
	$$
	v_k(x):=\bar v_k(x)+\frac1k |x|^2.
	$$
	Then clearly $v_k\in\fconvf\cap C^2_+(\R^n)$, and $v_k$ converges to $v$. 
 Moreover, if $\bar r>0$ is such that $v$ is of class $C^2$ in the set $\bar r\Bn$,
then the convergence is in $C^2$-norm in this set. 

Since $\oZZ{j}{\zeta}^*$ is continuous, it is now sufficient to prove that
$$
\lim_{k\to\infty}
\int_{\R^n}\zeta(\vert x\vert)\d\Phi^n_{j}(v_k,x)=\int_{\R^n}\zeta(\vert x\vert)\d\Phi^n_{j}(v,x).
$$
By (\ref{weak_cont}) and since $\zeta$ is continuous on $(0,\infty)$,  we have 
\begin{equation}\label{LA 1}
\lim_{k\to\infty}
\int_{\{x\colon \vert x\vert\ge r\}}\zeta(\vert x\vert)\d\Phi^n_{j}(v_k,x)=\int_{\{x\colon\vert x\vert\ge r\}}\zeta(\vert x\vert)\d\Phi^n_{j}(v,x)
\end{equation}
for $r>0$ with $\Phi^n_j(v,\{v=r\})=0$. Since $\Phi^n_j(v,\cdot)$ is locally finite, (\ref{LA 1}) holds for a.e.\ $r>0$. On the other hand, for every $r<\bar r$ and for every $k$ we have
$$
\int_{\{x\colon \vert x\vert< r\}}\zeta(\vert x\vert)\d\Phi^n_{j}(v_k,x)=
\int_{\{x\colon \vert x\vert< r\}}\zeta(\vert x\vert)[\Hess v_k(x)]_j\d x
$$
and a corresponding relation for $v$. Using the convergence of $v_k$ to $v$ in $\bar r \Bn$ in $C^2$-norm, polar coordinates and the fact that for $\delta>0$, 
$$\int_0^{\delta} r^{n-1}\zeta(r)\d r$$
is finite,  we obtain that for every $\varepsilon>0$ there exists $0<\delta<\bar r$ such that
\begin{equation}\label{LA 2}
\Big|\int_{\{x\colon \vert x\vert< \delta\}}\zeta(\vert x\vert)[\Hess v(x)]_j\d x\Big|, \Big|\int_{\{x\colon \vert x\vert< \delta\}}\zeta(\vert x\vert)[\Hess v_m(x)]_j\d x\Big|
\le\varepsilon
\end{equation}
for every $k\in\N$. The conclusion follows from \eqref{LA 1} and \eqref{LA 2}.
\end{proof}

\goodbreak
An immediate consequence is the following result.

\begin{lemma}\label{preparatory lemma 4 d}
	Let $j\in\{0,\dots,n\}$ and $\zeta\in \Had{j}{n}$. We have 
	\begin{equation*}\label{identity 4 d}
	\oZZ{j}{\zeta}(u)=\int_{\R^n}\zeta(\vert y\vert)\d\Psi^n_{j}(u,y)
	\end{equation*}
	for $u^*\in\fconvfs$.
\end{lemma}

\goodbreak
\section{The Classification Result}
\label{se:classification_of_valuations}
In the proof of Theorem \ref{thm:hadwiger_convex_functions} and Theorem \ref{dthm:hadwiger_convex_functions}, we follow the original approach by Hadwiger \cite{Hadwiger:V}. First, we establish a classification of valuations on $\fconvs$ that are epi-homogeneous of degree 1. For this, we introduce rotational epi-symmetrization and reduce the classification problem to the solution of an integral equation. We immediately obtain, by duality, a classification of valuations on $\fconvf$ that are homogeneous of degree 1. Next, we  introduce ortho\-gonal cylinder functions on $\fconvs$ and show that if a valuation vanishes on this class of functions, then it has to be epi-homogeneous of degree 1. By duality, we also obtain that if a valuation vanishes on dual orthogonal cylinder functions on $\fconvf$, then it has to be homogeneous of degree 1.  In the final step, we work in the dual setting and use induction on the dimension to establish a representation formula of valuations on dual orthogonal cylinder functions. Combined with the results established in the first part and the classification of one-homogeneous valuations, it completes the proof in the dual setting, that is, the proof of Theorem \ref{dthm:hadwiger_convex_functions}. The result in the primal setting, that is, Theorem \ref{thm:hadwiger_convex_functions}, follows by duality.

\subsection{The Case of Epi-Homogeneity of Degree 1}\label{se:one_hom_case}

Let $n\ge 2$.
A classical result by Hadwiger \cite[\S  4.5.3]{Hadwiger:V} (or see, \cite[Theorem 3.3.5]{Schneider:CB2}) states that for every convex body $K\subset \R^n$, that is at least one-dimensional, there exists a sequence of rotation means of $K$ that converges to a ball. Here a {\em rotation mean} of $K$ is given, for $k\ge 1$ and rotations $\vartheta_1, \dots, \vartheta_k\in \SO(n)$, as
$$\frac1k \left( \vartheta_1\, K+ \dots + \vartheta_k\, K\right).$$
For $u\in\fconvs$ and rotations $\vartheta_1, \dots, \vartheta_k\in \SO(n)$, we call
$$\frac1k \sq \left( u\circ \vartheta_1^{-1} \infconv \cdots \infconv u\circ \vartheta_k^{-1} \right) $$
a \emph{rotation epi-mean} of $u$. A function $u\in\fconvs$ is called {\em radial} if $u=u\circ\vartheta^{-1}$ for every $\vartheta\in\SO(n)$.

For $u\in\fconvs$, define its \emph{rotational epi-symmetrization} $u^\star$ via the support function of its epi-graph for $y\in\R^n$ and $t\in\R$  by
$$h_{\epi u^\star}(y,t):= \int\nolimits_{\SO(n)} h_{\epi u}(\vartheta^{-1} y,t)\d \vartheta,$$
where we integrate with respect to the Haar probability measure on $\SO(n)$.
The following result is obtained by suitably approximating the above integral.

\begin{lemma}
	\label{epi}
	For every $u\in \fconvs$ with at least one-dimensional domain, there is a sequence of rotation epi-means of $u$ epi-converging to its rotation epi-mean $u^\star\in\fconvs$.
\end{lemma}

Recall that for given $t\ge0$, we have $u_t(x):=t\vert x\vert+\ind_{\Bn}(x)$
for $x\in\R^n$. 

\begin{lemma}\label{ut}
If $u\in\fconvs$ is  radial, non-negative, and such that $u(0)=0$, then there exists a sequence $u_k$ which epi-converges to $u$ and such that each $u_k$ is the finite epi-sum of  functions 	of the form $r\sq u_t$ with $r>0$ and $t\ge 0$. 
\end{lemma} 

\begin{proof}
Since the function $u$ is radial, this is a one-dimensional problem. Any piecewise affine function $\bar u:[0,\infty)\to[0,\infty)$ with positive derivative at $0$ can be written as an epi-sum of finitely many functions of the form $s\mapsto t\,s +\ind_{[0,r]}(s)$ with suitable $r>0$ and $t\ge 0$. Hence the statement follows from the fact that such piecewise affine functions are dense.
\end{proof}

\goodbreak
The main result of this section is the following result.

\begin{proposition}
	\label{prop:class_1-hom} 
	Let $n\ge 2$.
	If $\,\oZ\colon\fconvs\to\R$ is a continuous, epi-translation  and rotation invariant valuation that is epi-homogeneous of degree 1, then there exists $\zeta \in \Had{1}{n}$ such that
	\begin{equation*}
	\oZ(u)=\oZZ{1}{\zeta}(u)
	\end{equation*}
	for every $u\in\fconvs$. 
\end{proposition}

\begin{proof}
	Since $\oZ$ is a continuous valuation that is epi-homogeneous of degree 1, Lemma \ref{epi-additve} implies that it is epi-additive.  Let $u\in\fconvs$ and  let $u_k$ be any rotation epi-mean of $u$, that is,  there are rotations $\vartheta_{1}, \dots, \vartheta_{m_k}$ such that
	$$u_k=\frac1{m_k} \sq \big( u\circ\vartheta_{1}^{-1} \infconv \cdots \infconv u\circ \vartheta_{m_k}^{-1}\big).$$
	Since $\oZ$ is epi-additive, rotation invariant and epi-homogeneous of degree 1, we see that
	$$\oZ(u)= \oZ(u_k).$$
	By Lemma \ref{epi}, there is a sequence of rotation epi-means of $u$ that converges to the rotation epi-symmetral $u^\star$. Since $\oZ$ is continuous, it follows  that
	$$
	\oZ(u)= \oZ(u^\star).
	$$
    Hence we obtain that $\oZ$ is determined by its values on radial functions. 

    Since $\oZ$ is epi-translation invariant, $\oZ$ is even determined by radial functions $u$ that are non-negative with $u(0)=0$. 
Combined with the epi-additivity and the epi-homogeneity of degree 1 of $\oZ$, Lemma \ref{ut} implies that $\oZ$ is already determined by its values on  the one-parameter family $u_t$ for $t\ge 0$.	Thus, the function $t\mapsto\oZ(u_t)$ on $[0,\infty)$ uniquely determines $\oZ$ and  it suffices to show that there is $\zeta\in\Had{1}{n}$ such that
\begin{equation}\label{ut equal}
\oZ(u_t)=\oZZ{1}{\zeta}(u_t)
\end{equation}
for $t\ge 0$.

Note that the continuity of $\oZ$ implies that $t\mapsto\oZ(u_t)$ is continuous in $[0,\infty)$.	Let us prove that it has compact support. Assume that this is not true. Hence, there exists
	a sequence $t_k$ such that $t_k\to\infty$ as $k\to\infty$, and 
	$$
	\alpha_k:=\oZ(u_{t_k})>0
	$$
	for all $k\in\N$.
	Define $\bar u_k\in\fconvs$  by
	$$
	\bar u_k(x):=\frac1{\alpha_k} \sq u_{t_k}(x)=t_k\vert x\vert+\ind_{\frac1{\alpha_k}B^n}(x).
	$$
	By the epi-homogeneity of $\oZ$, we have $\oZ(\bar u_k)=1$ for every $k\in\N$. On the other hand, the sequence $\bar u_k$ epi-converges to
	$\ind_{\{0\}}$, and it is easy to prove, using once more the fact that $\oZ$ is epi-homogeneous of degree 1, that $\oZ(\ind_{\{0\}})=0$. Hence we have 
	a contradiction to the continuity of $\oZ$.

To prove \eqref{ut equal}, we solve an integral equation. Let $u_t^*$ be the dual function of $u_t$ and note that we have $u_t^*\in\fconvfs$. By the definition of $\oZZ{1}{\bar\zeta}$ combined with Lemma \ref{preparatory lemma 4 d}, we can calculate $\oZZ{1}{\bar\zeta}(u_t)$ for $\bar\zeta\in\Had{1}{n}$ by Lemma~\ref{le:calc_ind_bn_tx_theta_i}  and obtain
	\begin{equation*}
	\oZZ{1}{\bar\zeta}(u_t)=\omega_n\Big(\bar\zeta(t)\,t^{n-1}+ (n-1) \,\int_t^\infty r^{n-2} \bar\zeta(r)\d r\Big).
	\end{equation*}
	Hence, we need to solve the integral equation,
	$$\zeta(t)+ \frac{n-1}{t^{n-1}}\int_t^\infty r^{n-2} \zeta(r)\d r =\frac{\oZ(u_t)}{\omega_n t^{n-1}}, $$
	to determine $\zeta$ when $\oZ(u_t)$ is given. We claim that
	\begin{equation}\label{solution}
	\zeta(t): = \frac{\oZ(u_t)}{\omega_n t^{n-1}} - \frac{(n-1)}{\omega_n}   \int_t^\infty \frac{\oZ(u_r)}{r^{n}}\d r
	\end{equation}
	is a solution.  
	Indeed, using integration by parts and that $t\mapsto \oZ(u_t)$ has compact support, we obtain that
	\begin{equation}\label{integral equation}
	(n-1) \int_t^{\infty} r^{n-2} \int_r^{\infty} \frac{\oZ(u_s)}{s^n} \d s \d r = - t^{n-1} \int_t^{\infty} \frac{\oZ(u_r)}{r^n} \d r +  \int_t^{\infty} \frac{\oZ(u_r)}{r} \d r,
	\end{equation}
	and the result follows.
	Since $t\mapsto \oZ(u_t)$ is continuous with compact support, also $\zeta$ is continuous on $(0, \infty)$ and it has bounded support. 
	
	Finally, we show that the solution $\zeta$ defined in (\ref{solution}) is in $\Had{1}{n}$. First, we show that  
	\begin{equation}\label{exists_finite}
    \lim_{t\to 0^+} \int_{t}^{\infty} r^{n-2} \zeta(r) \d r\qquad \mbox{exists and is finite.}
    \end{equation}
    Note, that for $t>0$,
	\begin{align*}
	\omega_n \int_{t}^{\infty}  r^{n-2} \zeta(r) \d r &= \int_t^{\infty} \left(\frac{\oZ(u_r)}{r} - (n-1) r^{n-2} \int_r^{\infty} \frac{\oZ(u_s)}{s^n} \d s \right) \d r
	\end{align*}
	and therefore, by (\ref{integral equation}),
	\begin{equation*}
		\omega_n \int_t^{\infty}  r^{n-2} \zeta(r) \d r = \int_t^{\infty} \frac{\oZ(u_r)}{r} \d r + t^{n-1} \int_t^{\infty} \frac{\oZ(u_r)}{r^n} \d r - \int_t^{\infty} \frac{\oZ(u_r)}{r} \d r= t^{n-1} \int_t^{\infty} \frac{\oZ(u_r)}{r^n} \d r.\\
	\end{equation*}
	Since $r\mapsto \oZ(u_r)$ is in $C_c([0,\infty))$ it follows from L'Hospital's rule that
    \begin{equation}\label{oml}
    \lim_{t\to 0^+} t^{n-1} \int_t^{\infty} \frac{\oZ(u_r)}{r^{n}} \d r =\frac{\oZ(u_0)}{n-1}
    \end{equation}
	and we have shown (\ref{exists_finite}).
	
	Second, since
	\begin{align*}
	\omega_n t^{n-1} \zeta(t) &= \omega_n t^{n-1} \left(\frac{\oZ(u_t)}{\omega_n t^{n-1}}  - \frac{n-1}{\omega_n} \int_t^{\infty} \frac{\oZ(u_r)}{r^n} \d r \right)=\oZ(u_t) - (n-1)t^{n-1} \int_t^{\infty} \frac{\oZ(u_r)}{r^n} \d r,
	\end{align*}
	it follows from \eqref{oml} that
	\begin{align*}
	\lim_{t\to 0^+} t^{n-1} \zeta(t)=0.
	\end{align*}
	Thus, $\zeta\in \Had{1}{n}$ and we have proved \eqref{ut equal}. 
 \end{proof}
 
 \goodbreak
 We immediately obtain the following dual result.
 
 \begin{proposition}
	\label{prop:class_1-hom dual} 
	Let $n\ge 2$.
	If $\,\oZ\colon\fconvf\to\R$ is a continuous, dually epi-translation  and rotation invariant valuation that is homogeneous of degree 1, then there exists $\zeta \in \Had{1}{n}$ such that
	\begin{equation*}\
	\oZ(v)=\oZZ{1}{\zeta}^*(v)
	\end{equation*}
	for every $v\in\fconvf$. 
\end{proposition}

\subsection{Orthogonal Cylinder Functions}
Let $P\subset \R^n$ be an $n$-dimensional convex polytope. We call the $n$-dimensional convex polytopes $P_1,\dots,P_m\subset \R^n$  a {\em dissection} of $P$, if  
$$P=\bigcup_{i=1}^m P_i$$ and
the interiors of the polytopes $P_i$ and $P_j$ are disjoint for $i\ne j$. In this case, we write 
$$P=\Ecup_{i=1}^m P_i.$$
Two $n$-dimensional convex polytopes $P$ and $Q$ are \emph{translatively equi-dissectable}, written $P \sim Q$, if there are dissections  $P=\Ecup_{i=1}^m P_i$ and $Q=\Ecup_{i=1}^m Q_i$ such that $P_i$ is a translate of $Q_i$ for $i=1,\dots, m$.

A valuation $\oZ: \cP^n\to \R$ is \emph{simple} if it vanishes on lower dimensional polytopes. For a simple valuation $\oZ: \cP^n\to\R$, we have
\begin{equation}\label{finite_add}
\oZ\big(\Ecup_{i=1}^{m} P_i\big)= \sum_{i=1}^{m} \oZ(P_i)
\end{equation}
(see, for example, \cite[Section 6.2]{Schneider:CB2}).

An $n$-dimensional simplex $S$ is the convex hull of $(n+1)$ affinely independent points $p_0, \dots, p_n$. Set $x_i= p_i-p_{i-1}$. We say that $S$ is orthogonal if the vectors $x_1, \dots, x_n$ are pairwise orthogonal. We set $x_0=p_0$ and write 
$S=\langle x_0;x_1,\dots,x_n \rangle$. 
For $0<t<1$,  the \emph{canonical simplex dissection} is 
\begin{equation}\label{css}
S=\Ecup_{k=0}^n ( (1-t)\, \underline S_k+ t\, \overline S_{n-k}),
\end{equation}
where 
$$\underline S_k:=\langle x_0;x_1,\dots,x_k\rangle,\qquad\overline S_{n-k}:=\langle x_0+\sum_{i=1}^k x_i;x_{k+1},\dots,x_n\rangle
$$
while $+$ in (\ref{css}) denotes Minkowski addition (see Hadwiger \cite{Hadwiger:V}).
\goodbreak

Let $\cP^n$ denote the set of convex polytopes in $\R^n$. We say that $P\in\cP^n$ is a {\em proper orthogonal cylinder} if there are orthogonal and complementary subspaces $E$ and $F$ with $\dim E, \dim F \ge 1$ and  polytopes $P_E\subset E$ and $P_F\subset F$ such that $P=P_E+P_F$. If $S$ is an orthogonal simplex, then the terms in (\ref{css}) for $1<k<n$ are proper orthogonal cylinders.

\goodbreak
The following result is due to Hadwiger  \cite[Section 1.3.4, Theorem VII]{Hadwiger:V}.

\begin{lemma}\label{ortho}
	For an $n$-dimensional polytope $P\in\cP^n$,  there are orthogonal simplices $S_1,\dots, S_{m}$, $S'_1,\dots, S'_{m'}$ such that
	$$
	P\ecup\Ecup_{i=1}^{m} S_i\sim \Ecup_{j=1}^{m'} S_j'.
	$$
\end{lemma}

Let $\oZ:\fconvs\to\R$ be a valuation. We say that it is \emph{simple} if $\oZ$ vanishes on functions with lower dimensional domain. The following analog of (\ref{finite_add}) is a consequence of the inclusion-exclusion principle established  in \cite{Colesanti-Ludwig-Mussnig-4}. Let $u_1,\dots, u_m\in\fconvs$ be such that the domains of $u_i\vee u_j$ are lower dimensional for $i\ne j$. If $\oZ: \fconvs\to\R$  is a simple valuation, then 
\begin{equation}\label{finite_add_fun}
\oZ\big(\bigwedge_{i=1}^{m} u_i\big)= \sum_{i=1}^{m} \oZ(u_i).
\end{equation}
A function $u\in\fconvs$ is called {\em piecewise affine} if its domain is a convex polytope and $u$ is the pointwise maximum of finitely many affine functions on its domain. Hence there is a dissection of $\dom(u)$ into polytopes $P_1,\dots,P_m$ such that
\begin{equation*}
u=\bigwedge_{i=1}^m(\ell_i+\alpha_i+\ind_{P_i})
\end{equation*}
with linear functions $\ell_i:\R^n \to \R$ and $\alpha_i\in\R$ for $i=1,\dots, m$. If $\oZ:\fconvs\to\R$ is an epi-translation invariant and simple valuation, we obtain by (\ref{finite_add_fun}) that
\begin{equation}\label{affine}
\oZ(u)= \sum_{i=1}^m \oZ(\ell_i+\ind_{P_i}).
\end{equation}
Note that piecewise affine functions are dense in $\fconvs$.

We say that the function $u\in\fconvs$ is an \emph{orthogonal cylinder function} if there are orthogonal and complementary subspaces $E$ and $F$ with $\dim E, \dim F\ge 1$ and functions $u_E\in\fconvs$ with $\dom (u_E)\subset E$ and $u_F\in\fconvs$ with $\dom (u_F)\subset F$ such that $u=u_E\infconv u_F$.

\begin{proposition}\label{onehom}
If $\,\oZ\colon\fconvs\to\R$ is a continuous and epi-translation invariant valuation that vanishes on orthogonal cylinder functions,   
then it is epi-homogeneous of degree 1.
\end{proposition}

\begin{proof}
   	For $y\in \R^n$, set $\ell_y(x)=\langle y,x\rangle$. Define $\tilde\oZ:\cP^n\to \R$  by
	$$
	\otZ(P):= \oZ(\ell_{y}+\ind_P).
	$$ 
	Since $\oZ$ is continuous, so is $\otZ$.
	
	Let $S$ be an $n$-dimensional orthogonal simplex in $\R^n$ and $0<t<1$. In the canonical simplex dissection
	$$
	S=\Ecup\limits_{k=0}^n ((1-t)\underline S_k+t\, \overline S_{n-k}),
	$$
	the simplices $\underline S_k$ and $\overline S_{n-k}$ are orthogonal and lie in orthogonal subspaces. 
	Since $\oZ$ vanishes on ortho\-gonal cylinder functions, it is simple. Consequently, $\otZ_y$ vanishes on orthogonal cylinders and is simple. By (\ref{finite_add}), this implies
	$$
	\otZ(S)= \otZ((1-t)S)+ \otZ(t\,S).
	$$
	Let $r,s>0$. Setting $\alpha(r)=\otZ(r\,S)$ and $t=r/(r+s)$, we obtain 
	$$\alpha(r+s)=\alpha(r)+\alpha(s)$$
	for $r,s>0$. Since $\otZ$ is continuous, so is $\alpha:(0,\infty)\to \R$. Hence, $\alpha$ is a continuous solution of Cauchy's functional equation, and we obtain that
	$$
	\otZ(t \,S)=t\otZ(S)
	$$
	for any $t>0$ and any orthogonal simplex $S$.
	
	\goodbreak
	
	For $P\in\cP^n$,  by Lemma \ref{ortho} there are orthogonal simplices such that
	$$
	P\ecup\Ecup_{i=1}^m S_i\sim \Ecup_{j=1}^{m'} S_j'.
	$$
	Therefore, for every $t>0$,
	\begin{equation*}
	\otZ(t\, P) = \otZ\Big(\Ecup_{j=1}^{m'} t\,S_j'\Big)-\otZ_y\Big(\Ecup_{i=1}^m t\,S_i\Big)
	=t \otZ\Big(\Ecup_{j=1}^{m'} S_j'\Big)-t \otZ_y\Big(\Ecup_{i=1}^m S_i\Big)
	= t\otZ(P).
	\end{equation*}
	Hence, $\otZ$ is homogeneous of degree 1 on polytopes.  Consequently, $\oZ$ is epi-homogeneous of degree 1 on functions of the form
	$$
	u=\ell_{y}+\ind_P
	$$
	for every $y\in\R^n$ and $P\in\cP^n$. Since $\oZ$ is epi-translation invariant and simple, (\ref{affine}) now implies that $\oZ$ is
	epi-homogeneous of degree 1 on piecewise affine functions. Since piecewise affine functions are dense in $\fconvs$, the continuity of $\oZ$ now implies that $\oZ$ is 	epi-homogeneous of degree 1 on $\fconvs$.
\end{proof}

We say that  $v\in\fconvf$ is a \emph{dual orthogonal cylinder function} if there are orthogonal and complementary subspaces $E$ and $F$ with $\dim E, \dim F\ge 1$ as well as  functions $v_E\in\fconvfE$ and $v_F\in\fconvfF$ such that $v=v_E+ v_F$. We use (\ref{infconvEF}) and obtain the following result as an immediate consequence of Proposition \ref{onehom}.

\begin{proposition}\label{onehom dual}
	If $\,\oZ\colon\fconvf\to\R$ is a  continuous and dually epi-translation invariant valuation that vanishes on dual orthogonal cylinder functions, then  
	it is homogeneous of degree 1.
\end{proposition}

\goodbreak
\subsection{Additional Results on Hessian Measures and Valuations}

\begin{lemma}
\label{le:decompose_hessian_measure}
Let $E$ and $F$ be orthogonal and complementary subspaces of $\R^n$ such that $1\le k<n$ with $k=\dim E$. If $v_E\in\fconvfE$ and $v_F\in\fconvfF$, then 
$$\Phi_l^n(v_E+v_F,B)=\sum_{i=0\vee(l+k-n)}^{k\wedge l} \Phi_i^k(v_E,B\cap E) \Phi_{l-i}^{n-k}(v_F,B\cap F)$$
for every $0\leq l \leq n$ and every Borel subset $B\subseteq \R^n$.
\end{lemma}
\begin{proof}
Let $v=v_E+v_F$. Observe that $y\in\partial v(x)$ if and only if $y_E\in\partial v_E(x_E)$ and $y_F \in\partial v_F(x_F)$. Therefore, for any $s\geq 0$ and any Borel subset $B\subseteq \R$,
\begin{align*}
P_s(v,B\times \R^n)&=\{x+sy\colon x\in B, y\in \partial v(x)\}\\
&=\{(x_E,x_F)+s(y_E,y_F)\colon x_E \in B\cap E, x_F\in B\cap F, y_E\in\partial v_E(x_E), y_F\in\partial v_F(x_F)\}\\
&= P_s(v_E,(B\cap E)\times E)\times P_s(v_F,(B\cap F)\times F).
\end{align*}
Thus,
\begin{align*}
\sum_{l=0}^n s^l \Phi_{l}^n(v,B) &=\hm^n(P_s(v,B\times \R^n))\\
&= \hm^k(P_s(v_E,(B\cap E)\times E)) \, \hm^{n-k}(P_s(v_F,(B\cap F)\times F))\\
&= \sum_{i=0}^k s^i \Phi_{i}^k(v_E,B\cap E) \sum_{j=0}^{n-k} s^j \Phi_{j}^{n-k}(v_F,B\cap F)
\end{align*}
and consequently,
$$
\Phi_l^n(v,B) = \sum_{i=0 \vee (l-(n-k))}^{k\wedge l} \Phi_i^k(v_E,B\cap E)\,\Phi_{l-i}^{n-k}(v_F,B\cap F),
$$
which completes the proof.
\end{proof}

We repeatedly evaluate Hessian measures at piecewise affine functions. Fix $\bar{x}=(\bar{x}_1,\ldots,\bar{x}_n)\in\R^n$ and define $\bar{v}\in\fconvf$ as 
\begin{equation}
\label{eq:v_bar}
\bar{v}(x_1,\ldots,x_n):=\frac 12 \sum_{i=1}^n |x_i-\bar{x}_i|.
\end{equation}
Note, that if $\bar{x}_1,\ldots,\bar{x}_n\neq 0$, then $\bar{v}\in\fconvfs$. If $\bar{u}=\bar{v}^*\in\fconvs$, then
$$\bar{u}(x)=\ind_{[-\frac 12,\frac 12]^n}(x)+\langle \bar{x},x\rangle$$
for $x\in\R^n$. It is now easy to see by \eqref{eq:int_zeta_u_is_int_zeta_v} that
$$\int_{\R^n} \zeta(x) \d \Phi_n^n(\bar{v},x)=\int_{\R^n} \zeta(y) \d\Psi_n^n(\bar{u},y) = \int_{[-\frac 12,\frac 12]^n} \zeta(\bar{x}) \d x = \zeta(\bar{x})$$
for every continuous function $\zeta:\R^n\to \R$ with compact support (for the second equality, compare \cite[Section 10.4]{Colesanti-Ludwig-Mussnig-3}). Hence, we conclude that
\begin{equation}
\label{eq:phi_n_n_delta}
\Phi_n^n(\bar{v},\cdot)=\delta_{\bar{x}},
\end{equation}
where $\delta_{\bar{x}}$ denotes the Dirac point measure concentrated at $\bar{x}$.

Next, for $1\leq j < n$, let 
$$v(x_1,\ldots,x_n):=\frac 12 \sum_{i=1}^j |x_i-\bar{x}_i|$$ 
with $\bar{x}_1,\ldots,\bar{x}_j\in\R$. Observe that $v$ is of the form $$v(x_1,\ldots,x_n)=v_1(x_1,\ldots,x_j)+v_2(x_{j+1},\ldots,x_n)$$ 
where $v_1$ is of the same form as \eqref{eq:v_bar} and $v_2\equiv 0$. In particular, $\Phi_i^{n-j}(v_2,A)=0$ for every $1\leq i \leq n-j$ and Borel subset $A\subseteq \R^{n-j}$. Hence, it follows from Lemma~\ref{le:decompose_hessian_measure}, \eqref{eq:phi_0_n} and \eqref{eq:phi_n_n_delta} that
\begin{align}
\begin{split}
\label{eq:phi_j_n_delta}
\d \Phi_j^n(v,(x_1,\ldots,x_n))&=\d\Phi_j^j(v_1,(x_1,\ldots,x_j))\d\Phi_0^{n-j}(v_2,(x_{j+1},\ldots,x_n))\\
&=\d\delta_{(\bar{x}_1,\ldots,\bar{x}_j)}((x_1,\ldots,x_j))\d x_{j+1}\cdots\d x_n
\end{split}
\end{align}
and
\begin{equation}
\label{eq:phi_i_n_zero}
\Phi_i^n(v,B)=0
\end{equation}
for every $j<i\leq n$ and Borel subset $B\subseteq \R^n$.

\begin{lemma}\label{preparatory lemma 1}
	Let $\zeta\in C_b((0,\infty))$  and $j\in\{1,\dots,n\}$. If
	\begin{equation*}
	\int_{\R^n}\zeta(\vert x\vert)\d\Phi^n_{j}(v,x)=0
	\end{equation*}
	for every $v\in\fconvfs$, then $\zeta\equiv0$. 	
\end{lemma}

\begin{proof}
	Fix  $\bar{x}_1,\ldots,\bar{x}_j\in\R\backslash\{0\}$, set $t=\sqrt{\bar{x}_1^2+\cdots +\bar{x}_j^2}\neq 0$ and ${v}(x_1,\ldots,x_n)=\frac 12\sum_{i=1}^j  |x_i-\bar{x}_i|$. For $j=n$, the measure $\Phi_{n}^n({v},\cdot)$ is the Dirac point mass concentrated at $(\bar{x}_1,\ldots,\bar{x}_n)$ by \eqref{eq:phi_n_n_delta}. Hence
	$$
	0=\int_{\R^n}\zeta(\vert x\vert)\d\Phi_{j}^n({v},x)=\zeta(t),
	$$
	and the result follows. For $1\leq j <n$, we use \eqref{eq:phi_j_n_delta} and obtain
	\begin{align*}
	0&=\int_{\R^n}\zeta(\vert x\vert)\d\Phi_{j}^n({v},x)\\
	&= \int_{\R^{n-j}} \zeta(\vert(\bar{x}_1,\ldots,\bar{x}_j,x_{j+1},\ldots,x_n)\vert) \d x_{j+1}\cdots \d x_n\\
	&= \omega_{n-j} \int_0^{\infty} \zeta(\sqrt{r^2+t^2})r^{n-j-1} \d r.
	\end{align*}
	Since $\bar{x}_1,\ldots,\bar{x}_j\in\R\backslash\{0\}$ and hence $t>0$ are arbitrary, the result now follows from Lemma~\ref{lemma Abel}.
\end{proof}

\begin{lemma}\label{preparatory lemma 2}
	Let $j\in\{1,\dots,n\}$ and $\zeta_m,\zeta\in \Had{j}{n}$. If
	\begin{equation}\label{identity 2}
	\lim_{m\to\infty}\int_{\R^n}\zeta_m(\vert x\vert)\d\Phi^n_{j}(v_m,x)=\int_{\R^n}\zeta(\vert x\vert)\d\Phi^n_{j}(v,x)
	\end{equation}
	for every $v_m, v\in\fconvfs$ with $v_m$ epi-convergent to $v$, then $\lim_{m\to\infty}\zeta_m(t)=\zeta(t)$ for every $t>0$. 	
\end{lemma}

\begin{proof}
	In case $j=n$, we may simply choose $v_m(x)=\bar{v}(x)=\frac 12 \sum_{j=1}^n  |x_j-\bar{x}_j|$ with $\bar{x}_j\in \R\backslash\{0\}$. It now follows from \eqref{eq:phi_n_n_delta} and \eqref{identity 2} that $\lim_{m\to\infty}\zeta_m(t)=\zeta(t)$ for $t=\sqrt{\bar x_1^2 + \dots + \bar x_n^2}>0$.

	So let $1\le j <n$. For $t>0$, we will consider the functions $v_t$ defined in (\ref{vt def}). Note, that $v_t\equiv 0$ on $t \Bn$ and hence $v_t\in\fconvfs$. It follows from Lemma~\ref{le:calc_ind_bn_tx_theta_i} that
	$$
	\int_{\R^n}\zeta_{m}(|x|)\d\Phi^{n}_j(v_{t},x)=\kappa_n\binom{n}{j}\rho_m(t),
	$$
	where $\rho_m\in C_c([0,\infty))$ is defined by
	\begin{equation}
	\label{eq:rho_m_eta_m}
	\rho_m(t):=(n-j)\eta_m(t)-t\eta_m'(t)
	\end{equation}
    and $\eta_m\in C_c([0,\infty))$ is in turn defined by
	$$
	\eta_m(t):=\int_t^{\infty} r^{n-j-1}\zeta_m(r)\d r.
	$$
	We will denote by $\rho$ and $\eta$ the corresponding functions associated to $\zeta$ according to the same relations.
	For every $t> 0$ and for every sequence $t_m$, with $t_m>0$, converging to $t$, we have $v_{t_m} \to v_t$ and hence, by the assumptions of the present lemma,
	$$
	\lim_{m\to\infty}\kappa_n\binom{n}{j}\rho_m(t_m)
=\lim_{m\to\infty}\int_{\R^n}\zeta_{m}(|x|)\d\Phi^{n}_j(v_{t_m},x)=\int_{\R^n}\zeta(|x|)\d\Phi^{n}_j(v_{t},x)=\kappa_n\binom{n}{j}\rho(t).
	$$
	In other words, the sequence of functions $\rho_m$ converges uniformly to $\rho$ on compact subsets of $(0,\infty)$.

	Next, we will show that there exists $\bar t>0$ such that 
	\begin{equation}\label{rhom}
	\rho_m(t)=0\qquad\mbox{for every $t>\bar t$ and $m\in\N$},
	\end{equation}
	which means that the supports of the functions $\rho_m$ are uniformly bounded.
	Assume that there exists a sequence $t_m\to\infty$ such that
	$$
	\rho_m(t_m)>0\qquad\mbox{for every $m\in\N$}.
	$$
	For $m\in\N$, set 
	$$
	r_m:=\rho_m(t_m)^{-\frac1j}
	$$
	and 
	$$
	v_{m}:=r_m v_{t_m}.
	$$
	Since $v_{t_m}(x)=0$ whenever $|x|\leq t_m$, it follows that $v_m$ converges to $v\equiv0\in\fconvfs$ as $m\to\infty$. On the other hand, by homogeneity,
	$$
	\int_{\R^n}\zeta_{m}(|x|)\d\Phi_j^{n}(v_{m},x)=\kappa_n \binom{n}{j} r_m^j\rho_m(t_m) =\kappa_n \binom{n}{j}\qquad\text{for $m\in\N,$}
	$$
	which contradicts the convergence to $v\equiv 0$. Thus (\ref{rhom}) holds.
	
    Let $\bar t>0$ be such that $\rho_m(t)=0$ for every $t\ge\bar t$ and $m\in\N$. Since $\eta_m$ has compact support, there exists $t_m>\bar t$ such that $\eta_m(t_m)=0$. Moreover, $\eta_m$ solves the Cauchy problem, 
	\begin{align*}
	&(n-j)\eta_m(t)-t\eta'_m(t)=0\qquad \mbox{in $(\bar t,\infty)$},\\
	&\eta_m(t_m)=0.
	\end{align*}
	This implies that $\eta_m\equiv0$ in $(\bar t,\infty)$. As $\bar t$ is independent of $m$, we deduce that the supports of the functions $\eta_m$ are uniformly bounded as well.   
	
	Equation \eqref{eq:rho_m_eta_m} implies
	$$
	-\frac{\d}{\d t}\left(
	\frac{\eta_m(t)}{t^{n-j}}
	\right)=\frac{\rho_m(t)}{t^{n-j+1}}\qquad\mbox{for $m\in\N$ and $t>0$}.
	$$
	Therefore,
	$$
	-\frac{\d}{\d t}\left(
	\frac{\eta_m(t)}{t^{n-j}}
	\right)
	$$
	converges uniformly to
	$$
	\frac{\rho(t)}{t^{n-j+1}}
	$$
	on compact subsets of $(0,\infty)$. Set 
	$$
	\bar\eta_m(t):=\frac{\eta_m(t)}{t^{n-j}}
	$$
	for $t>0$ and $m\in\N$. As the supports of the functions $\eta_m$ are uniformly bounded, there exists $t_1>0$ such that $\bar\eta_m(t_1)=0$ and $\rho_m(t_1)=0$ for every $m\in\N$. We deduce that $\bar\eta_m$ converges to a primitive of 
	$$
	-\frac{\rho(t)}{t^{n-j+1}}
	$$
	which vanishes at $t_1$. But since 
	$$
	\frac{\eta(t)}{t^{n-j}}
	$$
	has the same properties, it follows that $\eta_m$ converges to $\eta$ uniformly on compact subsets of $(0,\infty)$.
    We also have, from the definition of $\eta_m$ and $\rho_m$,
    \begin{equation}
    \label{eq:zeta_m_rho_m_eta_m}
    t^{n-j} \zeta_m(t)=\rho_m(t)-(n-j)\eta_m(t)
    \end{equation}
    for every $t>0$. Analogously,
    $$
    t^{n-j}\zeta(t)=\rho(t)-(n-j)\eta(t)
    $$
    for every $t>0$. Hence, passing to the limit as $m\to\infty$ in \eqref{eq:zeta_m_rho_m_eta_m}, we obtain the statement of the lemma. 
\end{proof}

\subsection{Proof of Theorem \ref{dthm:hadwiger_convex_functions}}

For $\zeta_0\in\Had{0}{n}, \ldots, \zeta_n\in\Had{n}{n}$,  we obtain from Theorem~\ref{dual main one way}  that 
$$v\mapsto \oZZ{0}{\zeta_0}^*(v)+\cdots+\oZZ{n}{\zeta_n}^*(v)$$
defines a continuous, dually epi-translation and rotation invariant valuation on $\fconvf$. 

Conversely, let $\oZ:\fconvf\to\R$ be a continuous, dually epi-translation and rotation invariant valuation. We want to show that there are $\zeta_0\in\Had{0}{n}, \ldots, \zeta_n\in\Had{n}{n}$ such that
$$\oZ(v)= \oZZ{0}{\zeta_0}^*(v)+\cdots+\oZZ{n}{\zeta_n}^*(v)$$
for every $v\in\fconvf$. By Theorem~\ref{thm:mcmullen_cvx_functions dual}, there is a homogeneous decomposition of $\oZ$. Therefore, we may assume that $\oZ$ is homogeneous of degree $l$ with $l\in\{0,\dots,n\}$. Since the cases $l=0,1,n$ are settled in Theorem~\ref{thm:class_0-hom dual}, Proposition~\ref{prop:class_1-hom dual} and Theorem~\ref{thm:class_n-hom dual}, respectively,  we may assume that $2\le l\le n-1$.
In particular, we already have a complete classification for the case $n=2$. Therefore, we assume that $n\ge 3$. 

Recall that $v\in\fconvfs$ if $v\in\fconvf$ and $v$ is of class $C^2$ in a neighborhood of the origin. 
This space  is dense in $\fconvf$ and by  Lemma \ref{preparatory lemma 4}, we have for $\zeta\in\Had{l}{n}$,
\begin{equation}\label{new statement}
\oZZ{l}{\zeta}^*(v)=\int_{\R^n} \zeta(|x|) \d\Phi_l^{n}(v,x)
\end{equation}
for every $v\in\fconvfs$.
\goodbreak

We proceed by induction on the dimension $n$. By (\ref{new statement}), we may use the induction hypothesis in the following form. Recall that the statement is true in the two-dimensional case.

\begin{assumption}\label{induction assumption}
	Let $2\le k\le n-1$. If\, $\oZ:\fconvfk\to\R$ is a continuous, dually epi-translation and rotation invariant valuation, then there exist $\zeta_i\in\Had{i}{k}$, for $0\le i \le k$, such that
	$$
	\oZ(v)= \sum_{i=0}^k\int_{\R^n} \zeta_i(|x|) \d\Phi_i^{k}(v,x)
	$$
	for every $v\in\fconvfsk$. 
\end{assumption}

\goodbreak
The main ingredient of our proof is the following result (which we prove using the induction hypo\-thesis).

\begin{proposition}\label{prop:one_zeta}
	If $\,\oZ:\fconvf\to\R$ is a continuous, dually epi-translation and rotation invariant valuation that is homogeneous of degree $l$ with $2\le l\le n-1$, then
	there exists $\zeta\in\Had{l}{n}$ such that 
	$$\oZ(v_E+v_F)=\int_{\R^n} \zeta(|x|)\d\Phi_l^{n}(v_E+v_F,x)$$
	for every pair of orthogonal and complementary subspaces $E,F$ with $1\le \dim E <n$ and every pair of functions $v_E\in\fconvfsE$ and $v_F\in\fconvfsF$.
\end{proposition}

Before proving this result, we complete the proof of Theorem \ref{dthm:hadwiger_convex_functions}. Define
$$\otZ(v):=\oZ(v)-\oZZ{l}{\zeta}^*(v)$$
for $v\in\fconvf$. Note that $\otZ$ is a continuous, dually epi-translation and rotation invariant valuation. 
From Proposition \ref{prop:one_zeta} and the continuity of $\otZ$, we obtain that
	$$\otZ(v)=0$$
for all dual orthogonal cylinder functions $v\in\fconvf$.
Thus, it follows from Proposition~\ref{onehom dual} that $\tilde{\oZ}$ is homogeneous of degree $1$. Since by construction $\tilde{\oZ}$ is also homogeneous of degree $l>1$, this implies that $\tilde{\oZ}\equiv 0$ and thus,
	$$\oZ(v)=\oZZ{l}{\zeta}^*(v)$$
for every $v\in\fconvf$.   This completes the proof of Theorem \ref{dthm:hadwiger_convex_functions}.

\goodbreak
\subsection{Proof of Proposition \ref{prop:one_zeta}}

Since $\oZ$ is rotation invariant and the roles of $E$ and $F$ can be interchanged, we may assume that $E=\spa\{e_1,\ldots,e_k\}$, $F=\spa\{e_{k+1},\ldots,e_n\}$, and $\lceil \frac n2 \rceil \le k\le n-1$. Since $n\geq 3$, this also implies that $k\geq 2$.

\begin{lemma}
\label{le:zeta_i_v_F}
Let $v_F\in\fconvfF$. There exist $\zeta_{i,v_F}\in \Had{i}{k}$ for $0\leq i \leq k$   such that
\begin{equation}
\label{eq:z_ve_+_vf_=_sum_zeta_i_vf}
\oZ(v_E+v_F)=\sum_{i=0}^k \int_E \zeta_{i,v_F}(|x_E|)\d\Phi_i^k(v_E,x_E)
\end{equation}
for  every $v_E\in\fconvfsE$. Moreover, the map
$$v_F\mapsto \int_E \zeta_{0,v_F}(|x_E|)\d\Phi_{0}^k(v_E,x_E)$$
and the maps $v_F\mapsto \zeta_{i,v_F}(s)$ for $1\leq i \leq k$ and $s>0$ are continuous, dually epi-translation and $\O(n-k)$ invariant valuations for every $v_E\in\fconvfsE$.
\end{lemma}
\begin{proof}
    For $v_F\in \fconvfF$, the map
	$$v_E\mapsto \oZ(v_E+v_F)$$
	is a continuous, dually epi-translation and $\SO(k)$ invariant valuation on $\fconvfE$. Since $k\geq 2$, it follows from the induction hypothesis that there exist functions $\zeta_{i, v_F}\in \Had{i}{k}$ for $0\leq i \leq k$, depending on $v_F\in\fconvfF$, such that (\ref{eq:z_ve_+_vf_=_sum_zeta_i_vf}) holds
	for every $v_E\in \fconvfsE$.
	
	Note, that the map $v_E \mapsto \int_{E} \zeta_{i,v_F}(|x_E|)\d\Phi_i^{k}(v_E,x_E)$ is homogeneous of degree $i$. 
Since the map $v_F \mapsto \oZ(v_E+v_F)$ is a continuous,  dually epi-translation and $\SO(n-k)$ invariant valuation on $\fconvfF$ for every $v_E\in \fconvfsE$, it follows from \eqref{eq:z_ve_+_vf_=_sum_zeta_i_vf} and homogeneity that also
	$$v_F \mapsto \int_E \zeta_{i,v_F}(|x_E|) \d \Phi_i^{k}(v_E,x_E)$$
is a valuation with the same properties on $\fconvfF$ for every $v_E\in\fconvfsE$ and $0\leq i \leq k$. This map is actually $\O(n-k)$ invariant, since $\oZ$ is rotation invariant and since \eqref{eq:z_ve_+_vf_=_sum_zeta_i_vf} shows that the map $v_E\mapsto \oZ(v_E+v_F)$ is $\O(k)$ invariant.

	Fix $i$ with $1\leq i \leq k$ and let $v_{F,1},v_{F,2}\in\fconvfF$ be such that also $v_{F,1}\wedge v_{F,2}\in\fconvfF$. By the valuation property, we have
	$$\int_E \left( \zeta_{i,v_{F,1} \vee v_{F,2}}(|x_E|) + \zeta_{i,v_{F,1} \wedge v_{F,2}}(|x_E|) - \zeta_{i,v_{F,1}}(|x_E|) - \zeta_{i,v_{F,2}}(|x_E|) \right) \d \Phi_i^{k}(v_E,x_E)=0$$
	for every $v_E\in\fconvfsE$. Hence, it follows from Lemma~\ref{preparatory lemma 1} that $v_F\mapsto \zeta_{i,v_F}(s)$ defines a valuation on $\fconvfF$ for every $s>0$. Similarly, it can be seen that this valuation is dually epi-translation and $\O(n-k)$ invariant. To prove that it is also continuous, note that the map
	\begin{equation*}\label{added1}
	(v_E,v_F)\mapsto\int_E\zeta_{i,v_F}(|x_E|)\d \Phi_i^{k}(v_E,x_E)
	\end{equation*}
	is jointly continuous in the two variables $v_E\in\fconvfsE$ and $v_F\in\fconvfF$. Hence continuity follows from Lemma \ref{preparatory lemma 2}.
\end{proof}

\goodbreak
In order to avoid an unnecessary distinction of  cases,  we set $\Phi_j^{m} \equiv 0$ if $j>m$ or $j<0$.

\begin{lemma}
    \label{le:zeta_i_l-i}
	For $0\leq i \leq k$ with $(l+k-n) \leq i\le  l$, there exist functions $\zeta_{i,l-i}:(0,\infty)^2\to \R$ such that
	\begin{align*}
	\oZ(v_E+v_F)=&\int_F \int_E \zeta_{0,l}(|x_E|,|x_F|)\d\Phi_0^{k}(v_E,x_E)\d\Phi_l^{n-k}(v_F,x_F)\\
	&+ \sum_{i=1 \vee (l+k-n)}^{k\wedge l} \int_E \int_F \zeta_{i,l-i}(|x_E|,|x_F|)\d\Phi_{l-i}^{n-k}(v_F,x_F) \d\Phi_{i}^{k}(v_E,x_E)
	\end{align*}
	for every $v_E\in\fconvfsE$ and $v_F\in\fconvfsF$. Moreover, $\zeta_{0,l}(\cdot,t)\in \Had{0}{k}$ for $t>0$ and
	$$\int_E \zeta_{0,l}(|x_E|,\cdot) \d\Phi_0^k(v_E,x_E)\in \Had{l}{n-k}$$
for every $v_E\in\fconvfsE$,
	while, for  $i\geq 1$, we have $\zeta_{i,l-i}(s,\cdot)\in \Had{l-i}{n-k}$  for $s>0$ and
	$$
	\int_{F} \zeta_{i,l-i}(\cdot,|x_F|)\d\Phi_{l-i}^{n-k}(v_F,x_F) \in \Had{i}{k}
	$$
for every $v_F\in\fconvfsF$. For $i=l+k-n$ with $i\ge 1$, the statements even hold for every $v_F\in\fconvfF$.
\end{lemma}
\begin{proof}
    By Lemma~\ref{le:zeta_i_v_F} combined with the induction hypothesis and Corollary~\ref{cor:1d}, there exist functions $\tilde{\zeta}_{i,j,s}\in \Had{j}{n-k}$ for every $1\leq i \leq k$ and $0\leq j \leq n-k$, depending on $s>0$, such that
	$$\zeta_{i,v_F}(s) = \sum_{j=0}^{n-k} \int_F \tilde{\zeta}_{i,j,s}(|x_F|) \d \Phi_j^{n-k}(v_F,x_F)$$
for every $v_F\in\fconvfsF$. Since the function $s\mapsto \zeta_{i,v_F}(s)$ belongs to $\Had{i}{k}$ and each of the summands on the right-hand side is homogeneous of a different degree in $v_F$, it follows that also the functions
	$$s\mapsto \int_F \tilde{\zeta}_{i,j,s}(|x_F|)\d\Phi_j^{n-k}(v_F,x_F)$$
belong to $\Had{i}{k}$ for every $v_F\in\fconvfsF$ and $0\le j\le n-k$. Note, that since both $\zeta_{i,v_F}$ and the last integral for the case $j=n-k$ are well-defined for every $v_F\in\fconvfF$, the last statement also holds for every $v_F\in\fconvf$ in this case. We remark that this also follows from Theorem~\ref{thm:mcmullen_cvx_functions dual} and Corollary~\ref{cor:class_n-hom dual}.
	Hence, there exist functions $\zeta_{i,j}:(0,\infty)^2\to\R$ with $\zeta_{i,j}(s,\cdot)\in \Had{j}{n-k}$ for every $s>0$ and $\int_F \zeta_{i,j}(\cdot,|x_F|) \d\Phi_j^{n-k}(v_F,x_F)\in \Had{i}{k}$ for every $v\in\fconvfsF$ such that
	$$\tilde{\zeta}_{i,j,s}(t)=\zeta_{i,j}(s,t)$$
	for every $s, t>0$. Moreover, $\int_F \zeta_{i,n-k}(\cdot,|x_F|) \d\Phi_{n-k}^{n-k}(v_F,x_F)\in \Had{i}{k}$ for every $v\in\fconvfF$.

	For the case $i=0$, note that
	$$\int_E \zeta_{0,v_F}(|x_E|) \d \Phi_0^{k}(v_E,x_E)=\int_E \zeta_{0,v_F}(|x_E|) \d x_E = \tilde\zeta_{0,v_F}$$
is a constant that is independent of $v_E$. By Lemma~\ref{le:zeta_i_v_F} the map $v_F \mapsto \tilde\zeta_{0,v_F}$ is a continuous, dually epi-translation and $\O(n-k)$ invariant valuation on $\fconvfF$. Therefore, by the induction hypothesis,
there exist functions $\tilde{\zeta}_{0,j}\in\Had{j}{n-k}$ for every $0\leq j \leq n-k$ such that
	$$\tilde\zeta_{0,v_F} = \sum_{j=0}^{n-k}\int_F \tilde{\zeta}_{0,j}(|x_F|) \d \Phi_j^{n-k}(v_F,x_F)$$
	for every $v_F\in\fconvfsF$. It is now possible to find (non-unique) functions $\zeta_{0,j}:(0,\infty)^2\to\R$ such that $\zeta_{0,j}(\cdot,t)\in \Had{0}{k}$ for every $t>0$ and
	$$\tilde{\zeta}_{0,j}(t) = \int_E \zeta_{0,j}(|x_E|,t) \d x_E$$
	for every $t>0$ and $0\leq j \leq n-k$.

	Combining the cases $i>0$ and $i=0$, we obtain that
	\begin{align*}
	\oZ(v_E+v_F) = &\sum_{j=0}^{n-k} \int_F \int_E \zeta_{0,j}(|x_E|,|x_F|)\d \Phi_0^{k}(v_E,x_E)\d\Phi_j^{n-k}(v_F,x_F)\\
	&+\sum_{i=1}^k \sum_{j=0}^{n-k} \int_E \int_F \zeta_{i,j}(|x_E|,|x_F|) \d\Phi_j^{n-k}(v_F,x_F) \d\Phi_i^{k}(v_E,x_E)
	\end{align*}
	for every $v_E\in\fconvfsE$ and $v_F\in\fconvfsF$. As $\oZ$ is homogeneous of degree $l$, and $\Phi_{i}^k(v_E,\cdot)$ and $\Phi_{j}^{n-k}(v_F,\cdot)$ in the sum 
	on the right side of the previous equation are homogeneous of degree $i$ and $j$, respectively, only those terms for which $i+j=l$ may appear. Moreover, we must have
	$$
	i\le l,\quad i\le k,\quad l-i\le n-k
	$$
	which gives the desired representation of $\oZ$.
\end{proof}
\goodbreak

Recall that $\lceil \frac n2 \rceil \le k\le n-1$ and that $E=\spa\{e_1, \dots, e_k\}$ while $F=\spa\{e_{k+1}, \dots, e_n\}$.  
Set $E'=\spa\{e_1, \dots, e_{k-1}\}$ and $F'=\spa\{e_{k+2},\dots, e_n\}$. In particular, for $k=n-1$, we have $F'=\{0\}$ and $\fconvfsFp$ can be identified with $\R$. In order to simplify notation, we write $\int_{F'} \d \Phi_0^{0}(v,x) = 1$ for $v\in\fconvfsFp$.

\begin{lemma}
	\label{le:int_zeta_i_l-i_zeta_i+1_l-i-1}
Let $1\le i< k$ and $a\in \R\backslash\{0\}$. For $v_{E'}\in\fconvfsEp$ and $v_{F'}\in\fconvfsFp$,
	\begin{align*}
	    \int_{F'} \int_{E} &\zeta_{0,l}(|x_E|,|(a,x_{F'})|)\d x_E \d \Phi_{l-1}^{n-k-1}(v_{F'},x_{F'})\\
	    &= \int_{E'} \int_\R \int_{F'} \zeta_{1,l-1}(|(x_{E'},a)|,|(t,x_{F'})|) \d\Phi_{l-1}^{n-k-1}(v_{F'},x_{F'}) \d t \d x_{E'}
	\end{align*}
	and
	\begin{align*}
	\int_{E'} \int_\R \int_{F'}  &\zeta_{i,l-i}(|(x_{E'},s)|,|(a,x_{F'})|) \d\Phi_{l-i-1}^{n-k-1}(v_{F'},x_{F'})\d s \d\Phi_i^{k-1}(v_{E'},x_{E'})\\
	=& \int_{E'} \int_\R \int_{F'} \zeta_{i+1,l-i-1}(|(x_{E'},a)|,|(t,x_{F'})|) \d\Phi_{l-i-1}^{n-k-1}(v_{F'},x_{F'}) \d t \d\Phi_i^{k-1}(v_{E'},x_{E'})
	\end{align*}
where $(l+k-n)\leq i < l$.
\end{lemma}
\begin{proof}
		Let $\alpha, \beta\in \R$ and $t_\alpha, t_\beta\in \R\backslash\{0\}$. Set
	$$v_E(x_E)=v_{E'}(x_{E'})+\frac{\alpha}{2} |x_k-t_\alpha|$$
	for $v_{E'}\in\fconvfsEp$ and 
	$$v_F(x_F)=\frac{\beta}{2} |x_{k+1}-t_\beta| + v_{F'}(x_{F'})$$
	for $v_{F'}\in\fconvfsFp$. Using Lemma~\ref{le:decompose_hessian_measure}, \eqref{eq:phi_0_n} and \eqref{eq:phi_n_n_delta}, we obtain
	\begin{eqnarray*}
		\d\Phi_{i}^{k}(v_E,x_E) &=& \d\Phi_i^{k-1}(v_{E'},x_{E'})\d x_k + \alpha \d\Phi_{i-1}^{k-1}(v_{E'},x_{E'})\d\delta_{t_{\alpha}}(x_k),\\
		\d\Phi_{l-i}^{n-k}(v_F,x_F) &=& \d x_{k+1}\d\Phi_{l-i}^{n-k-1}(v_{F'},x_{F'}) + \beta\d\delta_{t_\beta}(x_{k+1}) \d\Phi_{l-i-1}^{n-k-1}(v_{F'},x_{F'})
	\end{eqnarray*}
	for $(l+k-n)\leq i \leq l$. Thus, by Lemma~\ref{le:zeta_i_l-i},
	\begin{align*}
	\oZ&(v_E+v_F)\\
	&= \int_F \int_E \zeta_{0,l}(|x_E|,|x_F|) \d\Phi_{0}^{k}(v_E,x_E)\d\Phi_{l}^{n-k}(v_F,x_F)\\
	&\quad\; + \sum_{i=1 \vee (l+k-n)}^{k\wedge l} \int_E \int_F \zeta_{i,l-i}(|x_E|,|x_F|) \d\Phi_{l-i}^{n-k}(v_F,x_F) \d \Phi_{i}^{k}(v_E,x_E)\\
	&=\int_{F'} \int_{\R} \int_{E}\zeta_{0,l}(|x_E|,|(x_{k+1},x_{F'})|) \d x_E \d x_{k+1} \d\Phi_{l}^{n-k-1}(v_{F'},x_{F'})\\
	&\quad\; + \beta \int_{F'} \int_{E}\zeta_{0,l}(|x_E|,|(t_{\beta},x_{F'})|) \d x_E \d\Phi_{l-1}^{n-k-1}(v_{F'},x_{F'})\\
	&\quad\; +\sum_{i=1 \vee (l+k-n)}^{k\wedge l}\bigg( \int_{E'} \int_\R \int_{F'} \int_\R \zeta_{i,l-i}(|(x_{E'},s)|,|(t,x_{F'})|) \d t \d\Phi_{l-i}^{n-k-1}(v_{F'},x_{F'}) \d s \d\Phi_{i}^{k-1}(v_{E'},x_{E'})   \\
	&\qquad\qquad\qquad +\alpha \int_{E'} \int_\R \int_{F'} \zeta_{i,l-i}(|(x_{E'},t_\alpha)|,|(t,x_{F'})|) \d\Phi_{l-i}^{n-k-1}(v_{F'},x_{F'}) \d t 
\d\Phi_{i-1}^{k-1}(v_{E'},x_{E'})  \\
	&\qquad\qquad\qquad +\beta \int_{E'} \int_\R \int_{F'}  \zeta_{i,l-i}(|(x_{E'},t)|,|(t_\beta,x_{F'})|) \d\Phi_{l-i-1}^{n-k-1}(v_{F'},x_{F'})\d t \d\Phi_i^{k-1}(v_{E'},x_{E'}) \\
	&\qquad\qquad\qquad + \alpha \beta \int_{E'} \int_{F'} \zeta_{i,l-i}(|(x_{E'},t_\alpha)|,|(t_\beta,x_{F'})|) \d\Phi_{l-i-1}^{n-k-1}(v_{F'},x_{F'})\d\Phi_{i-1}^{k-1}(v_{E'},x_{E'}) \bigg).
	\end{align*}
	Here we have split the integral over $E$ with respect to $\d\Phi_i^k(v_E,x_E)$ and the integral over $F$ with respect to $\d\Phi_{l-i}^{n-k}(v_F,x_F)$ into multiple integrals. We are allowed to do so since, for given $i$, each of the integrals so obtained is homogeneous of different degrees in $v_{E'}$ and $v_{F'}$. In particular, each of the integrals in the last expression is well-defined and finite. Since $\oZ$ is rotation invariant and this expression only depends on the absolute values of $t_\alpha$ and $t_\beta$, respectively, we may exchange $(\alpha,t_\alpha)$ and $(\beta,t_\beta)$ while preserving equality. Making this exchange, setting $t_\alpha=t_\beta$, comparing with the original expression and considering those parts which are $i$-homogeneous in $v_{E'}$ and $(l-i-1)$-homogeneous $v_{F'}$ gives
	\begin{align*}
	    \beta&\int_{F'} \int_{E} \zeta_{0,l}(|x_E|,|(t_\alpha,x_{F'})|) \d x_E \d\Phi_{l-1}^{n-k-1}(v_{F'},x_{F'})\\
	    &\quad +\alpha \int_{E'} \int_\R \int_{F'} \zeta_{1,l-1}(|(x_{E'},t_\alpha)|,|(t,x_{F'})|)\d\Phi_{l-1}^{n-k-1}(v_{F'},x_{F'}) \d t \d x_{E'}\\
	    =\alpha&\int_{F'} \int_{E} \zeta_{0,l}(|x_E|,|(t_\alpha,x_{F'})|) \d x_E \d\Phi_{l-1}^{n-k-1}(v_{F'},x_{F'})\\
	    &\quad +\beta\int_{E'} \int_\R \int_{F'} \zeta_{1,l-1}(|(x_{E'},t_\alpha)|,|(t,x_{F'})|)\d\Phi_{l-1}^{n-k-1}(v_{F'},x_{F'}) \d t \d x_{E'}
	\end{align*}
	for $i=0$ and
	\begin{align*}
	\beta&\int_{E'} \int_\R \int_{F'} \zeta_{i,l-i}(|(x_{E'},t)|,|(t_\alpha,x_{F'})|) \d\Phi_{l-i-1}^{n-k-1}(v_{F'},x_{F'}) \d t \d\Phi_i^{k-1}(v_{E'},x_{E'})\\
	&\quad +\alpha \int_{E'} \int_\R \int_{F'} \zeta_{i+1,l-i-1}(|(x_{E'},t_\alpha)|,|(t,x_{F'})|) \d\Phi_{l-i-1}^{n-k-1}(v_{F'},x_{F'}) \d t \d\Phi_i^{k-1}(v_{E'},x_{E'})\\
	=\alpha&\int_{E'} \int_\R \int_{F'} \zeta_{i,l-i}(|(x_{E'},t)|,|(t_\alpha,x_{F'})|)\d\Phi_{l-i-1}^{n-k-1}(v_{F'},x_{F'}) \d t \d\Phi_i^{k-1}(v_{E'},x_{E'})\\
	&\quad +\beta \int_{E'} \int_\R \int_{F'} \zeta_{i+1,l-i-1}(|(x_{E'},t_\alpha)|,|(t,x_{F'})|) \d\Phi_{l-i-1}^{n-k-1}(v_{F'},x_{F'}) \d t \d\Phi_i^{k-1}(v_{E'},x_{E'})
	\end{align*}
	for $i\geq 1$. The desired equality now follows after rearranging, using the fact that $\alpha$ and $\beta$ are arbitrary, and setting $a=t_\alpha$.
\end{proof}

\begin{lemma}
\label{le:zeta_i}
For $1\le i <k$  with $(l+k-n)\le i <l$, there exists $\zeta_i\in C_b((0,\infty))$ such that
$$\zeta_{i,l-i}(a,b)=\zeta_i(\sqrt{a^2+b^2})$$
for every $a,b>0$, and by continuity this extends to all $a,b\ge 0$ with $(a,b)\ne (0,0)$. Moreover, if $k=n-1$, then $\zeta_{l-1}\in \Had{l}{n}$.
\end{lemma}
\begin{proof}
	Since $i\ge 1$, it follows from Lemma~\ref{preparatory lemma 1} and Lemma~\ref{le:int_zeta_i_l-i_zeta_i+1_l-i-1} that
	$$
	\int_{\R} \int_{F'} \big(\zeta_{i+1,l-i-1}(|(a,b)|,|(t,x_{F'})|)	- \zeta_{i,l-i}(|(a,t)|,|(b,x_{F'})|)\big) \d\Phi_{l-i-1}^{n-k-1}(v_{F'},x_{F'}) \d t= 0
	$$
	for every $a,b>0$ and $v_{F'}\in\fconvfsFp$. Writing the same equation again but this time replacing $b$ by $\varepsilon$ with $0<\varepsilon<b$ and $a$ by $\sqrt{a^2+b^2-\varepsilon^2}$, we obtain
$$
	\int_{\R} \int_{F'} 
	\big(\zeta_{i+1,l-i-1}(|(a,b)|,|(t,x_{F'})|)	- \zeta_{i,l-i}(|(\sqrt{a^2+b^2-\varepsilon^2},t)|,|(\varepsilon,x_{F'})|)\big) 
	\d\Phi_{l-i-1}^{n-k-1}(v_{F'},x_{F'}) \d t= 0
$$
	for every $a>0$ and $b>\varepsilon>0$. By subtracting the last two equations, we obtain 
$$
	\int_{\R} \int_{F'} 
	\big(\zeta_{i,l-i}(|(\sqrt{a^2+b^2-\varepsilon^2},t)|,|(\varepsilon,x_{F'})|)- \zeta_{i,l-i}(|(a,t)|,|(b,x_{F'})|)\big) 
	\d\Phi_{l-i-1}^{n-k-1}(v_{F'},x_{F'}) \d t= 0	
$$
	for every $a>0$ and $b>\varepsilon>0$. By the properties of $\zeta_{i,l-i}$ we may apply Lemma~\ref{lemma Abel} for fixed $b>\varepsilon>0$ and obtain
	\begin{equation*}	
	\int_{F'} \big(\zeta_{i,l-i}(\sqrt{a^2+b^2-\varepsilon^2},|(\varepsilon,x_{F'})|)	- \zeta_{i,l-i}(a,|(b,x_{F'})|)\big) \d\Phi_{l-i-1}^{n-k-1}(v_{F'},x_{F'}) = 0			\end{equation*}
	for every $a>0$. Similarly to before, we write the last equation again, but this time, we replace $a$ by $\delta$ with $0<\delta<a$. Furthermore, we may replace $b$ by $\sqrt{a^2+b^2-\delta^2}$ since by our choice of $\delta$ we have $\sqrt{a^2+b^2-\delta^2}>b>\varepsilon$. Thus, we obtain
	\begin{equation*}	
\int_{F'} \big(\zeta_{i,l-i}(\sqrt{a^2+b^2-\varepsilon^2},|(\varepsilon,x_{F'})|)	- \zeta_{i,l-i}(\delta,|(\sqrt{a^2+b^2-\delta^2},x_{F'})|)\big) \d\Phi_{l-i-1}^{n-k-1}(v_{F'},x_{F'}) = 0	
\end{equation*}
	for every $a>\delta>0$ and $b>\varepsilon>0$. Subtracting the last two equations yields
	\begin{equation}	\label{eq:int_zeta_i_l-i_a_b_delta_no_eps}
	\int_{F'} \big(\zeta_{i,l-i}(\delta,|(\sqrt{a^2+b^2-\delta^2},x_{F'})|)
	- \zeta_{i,l-i}(a,|(b,x_{F'})|)\big) \d\Phi_{l-i-1}^{n-k-1}(v_{F'},x_{F'}) = 0    
	\end{equation}
	for every $a>\delta >0$ and $b>\varepsilon>0$. Since this expression does not depend on $\varepsilon$ anymore and $\varepsilon>0$ was arbitrary, we conclude that \eqref{eq:int_zeta_i_l-i_a_b_delta_no_eps} holds for every $b>0$.

	We now claim that
	\begin{equation}
	\label{eq:zeta_i_l-i_switch}
	\zeta_{i,l-i}(\sqrt{a^2+b^2-\delta^2},\delta) = \zeta_{i,l-i}(a,b)
	\end{equation}
	for every $a>\delta>0$ and  $b>0$. In case $l-i-1>0$, this is a direct consequence of Lemma~\ref{preparatory lemma 1}. 
In case $l-i-1=0$ and $k=n-1$, we have $F'=\{0\}$ and $\int_{F'} \d\Phi_{l-i-1}^{n-k-1}(v_{F'},x_{F'})= 1$ and the claim is immediate. In the remaining case $l-i-1=0$ and $k<n-1$, we want to emphasize that
	$$\d\Phi_{l-i-1}^{n-k-1}(v_{F'},x_{F'})=\d x_{F'}$$
	and therefore, \eqref{eq:zeta_i_l-i_switch} follows from Lemma~\ref{lemma Abel} after switching to spherical coordinates in \eqref{eq:int_zeta_i_l-i_a_b_delta_no_eps} and considering that $b>0$ is arbitrary.

	Note that the right side of \eqref{eq:zeta_i_l-i_switch} does not depend on the choice of $\delta<a$. Hence, there exists a function $\zeta_i:(0,\infty)\to \R$ such that
	\begin{equation}
	\label{eq:zeta_i_l-i_zeta_i}
	\zeta_{i,l-i}(a,b)=\lim_{\delta\to 0^+} \zeta_{i,l-i}(\sqrt{a^2+b^2-\delta^2},\delta) =:\zeta_i(\sqrt{a^2+b^2})
	\end{equation}
	for every $a,b>0$. Furthermore, it follows from the properties of $\zeta_{i,l-i}$ that $\zeta_i$ is continuous with bounded support. Moreover, for fixed $a>0$ we obtain
	$$\lim_{b\to 0^+}\zeta_{i,l-i}(a,b)=\lim_{b\to 0^+} \zeta_i(\sqrt{a^2+b^2}) = \zeta_i(\sqrt{a^2/2+a^2/2})=\zeta_{i,l-i}(a/\sqrt{2},a/\sqrt{2}).$$
	Thus, \eqref{eq:zeta_i_l-i_zeta_i} continuously extends to all $a,b\ge 0$ with $(a,b)\ne(0,0)$.

	Note that in case $k=n-1$ and $i=l-1$, we have $\zeta_{l-1,1}(s,\cdot)\in \Had{1}{1}$ for every $s>0$ and that $\int_F \zeta_{l-1,1}(\cdot,|x_F|) \d\Phi_1^1(v_F,x_F)\in \Had{l-1}{n-1}$ for every $v_F\in\fconvfF$. Thus, choosing $v_F(x_F)=\frac12 \sum_{i=k+1}^n |x_i|$ we obtain by \eqref{eq:phi_n_n_delta}
	$$\zeta_{l-1}(s)=\zeta_{l-1,1}(s,0)=\int_F \zeta_{l-1,1}(s,|x_F|) \d\Phi_1^1(v_F,x_F)$$
    for every $s>0$. Hence $\zeta_{l-1} \in\Had{l-1}{n-1}=\Had{l}{n}$.
\end{proof}

\begin{lemma}
    \label{le:zeta^k}
For $\lceil \frac n2 \rceil\leq k\leq n-1$, there exists $\zeta_{(k)}\in C_b((0,\infty))$  such that
    \begin{align*}
    \oZ(v_E+v_F)=&\int_F \int_E \zeta_{(k)}(\sqrt{|x_E|^2+|x_F|^2}) \d\Phi_0^k(v_E,x_E) \d\Phi_l^{n-k}(v_F,x_F)\\
    &+\sum_{i=1\vee(l+k-n)}^{k\wedge l} \int_E\int_F \zeta_{(k)}(\sqrt{|x_E|^2+|x_F|^2}) \d\Phi_{l-i}^{n-k}(v_F,x_F)\d\Phi_i^k(v_E,x_E)
    \end{align*}
    for every $v_E\in\fconvfsE$ and $v_F\in\fconvfsF$. Moreover, $\zeta_{(n-1)}\in\Had{l}{n}$.
\end{lemma}

\begin{proof}
By Lemma \ref{le:zeta_i_l-i} combined with Lemma \ref{le:zeta_i}, we have
\begin{align*}
	\oZ(v_E+v_F)=&\int_F \int_E \zeta_{0,l}(|x_E|,|x_F|)\d\Phi_0^{k}(v_E,x_E)\d\Phi_l^{n-k}(v_F,x_F)\\
	&+ \sum_{i=1 \vee (l+k-n)}^{(k\wedge l)-1} \int_E \int_F \zeta_{i}(\sqrt{|x_E|^2+|x_F|^2})\d\Phi_{l-i}^{n-k}(v_F,x_F) \d\Phi_{i}^{k}(v_E,x_E)\\
	&+ \int_E \int_F \zeta_{k\wedge l,l-(k\wedge l)}(|x_E|,|x_F|) \d\Phi_{l-(k\wedge l)}^{n-k}(v_F,x_F)\d\Phi_{k\wedge l}^k(v_E,x_E)
\end{align*}
for every $v_{E}\in\fconvfsE$ and $v_{F}\in\fconvfsF$ where $\zeta_i\in C_b((0,\infty))$ for every $i$ and, if $k=n-1$, in addition, $\zeta_{l-1}\in \Had{l}{n}$. We need to show that we can choose $\zeta_{1 \vee (l+k-n)}=\dots=\zeta_{(k\wedge l)-1}=:\zeta_{(k)}$ and that we can also write the first integral (in case it does not vanish) using $\zeta_{(k)}$. Moreover, in case $k<l$, our proof will show that $\zeta_{k,l-k}=\zeta_{(k)}$ and in case $l\leq k$, we will rewrite the last integral using $\zeta_{(k)}$.

	Fix $0\vee(l+k-n) \leq i < k\wedge l$. We will first consider the cases $i>0$ and $l-i-1>0$, and we will treat the cases $i=0$ and $i=l-1$ separately. Since $i>0$ it follows from Lemma~\ref{preparatory lemma 1}, Lemma~\ref{le:int_zeta_i_l-i_zeta_i+1_l-i-1} and Lemma~\ref{le:zeta_i} that
	\begin{equation}
	\label{eq:zeta_i+1_l-i-1_zeta_i_dphi_l-i-1_dt}
	\int_{\R} \int_{F'} \Big(\zeta_{i+1,l-i-1}(a,|(t,x_{F'})|)
	- \zeta_i(|(a,t,x_{F'})|)\Big) \d\Phi_{l-i-1}^{n-k-1}(v_{F'},x_{F'}) \d t= 0
	\end{equation}
    for every $a>0$ and $v_{F'}\in\fconvfsFp$. By \eqref{eq:phi_j_n_delta}, choosing $a>0$ and $v_{F'}(x_{F'})=\frac12(|x_{k+2}-\bar{x}_{k+2}|+\cdots+|x_{k+l-i}-\bar{x}_{k+l-i}|)$ with arbitrary $\bar{x}_{k+2}$,$\ldots$, $\bar{x}_{k+l-i}\in\R\backslash\{0\}$ gives
	\begin{multline*}
	\int_{\R} \int_{\R^{n+i-k-l}} \Big(\zeta_{i+1,l-i-1}(a,|(t,b,x_{k+l-i+1},\ldots,x_n)|)\\
	-\zeta_{i}(|(a,t,b,x_{k+l-i+1},\ldots,x_n)|)\Big) \d x_{k+l-i+1}\cdots \d x_n \d t=0
	\end{multline*}
	for every $b>0$. Hence, by Lemma~\ref{lemma Abel},
	$$\zeta_{i+1,l-i-1}(a,b)=\zeta_{i}(\sqrt{a^2+b^2})$$
	for every $a,b>0$ and similarly as in the proof of Lemma~\ref{le:zeta_i}, for every $a,b\ge 0$ with $(a,b)\ne(0,0)$. Note, that in case $i+1<k\wedge l$ it follows from Lemma~\ref{le:zeta_i} that also $\zeta_{i+1,l-i-1}(a,b)=\zeta_{i+1}(\sqrt{a^2+b^2})$ for every $a,b\ge 0$ with $(a,b)\ne(0,0)$. Thus, we have shown that there exists $\zeta_{(k)}\in C_b((0,\infty))$ such that
	$$\zeta_{(k)}(\sqrt{a^2+b^2}):=\zeta_{j,l-j}(a,b)$$
	for every $a,b \ge 0$ with $(a,b)\ne(0,0)$ and every $1\vee (l+k-n)\leq j \leq k \wedge (l-1)$, except for the cases $l=2$ or $k=n-1$ (for which only $i=0$ and/or $i=l-1$ satisfy the conditions on $i$ above).

	Observe that the case $i=l-1$ can only occur if $l\leq k$. Since $l\geq 2$, we have $i>0$, and therefore equation \eqref{eq:zeta_i+1_l-i-1_zeta_i_dphi_l-i-1_dt} still holds. We now claim that we may replace $\zeta_i=\zeta_{l-1}$ by $\zeta_{(k)}$ in \eqref{eq:zeta_i+1_l-i-1_zeta_i_dphi_l-i-1_dt}. If $k<n-1$ and $l>2$, we have $1\vee (l+k-n)\leq l-2$ and therefore this follows from the previous considerations for the case $i=l-2$. If $k=n-1$ or $l=2$ we simply set $\zeta_{(k)}:=\zeta_{l-1}$ by Lemma~\ref{le:zeta_i} which also shows that $\zeta_{(n-1)}\in\Had{l}{n}$. Thus, we obtain
	\begin{align*}
	0&=\int_\R \int_{F'} \big(\zeta_{l,0}(a,|(t,x_{F'})|)-\zeta_{(k)}(|(a,t,x_{F'})|) \big) \d x_{F'}  \d t\\
	&= \int_F \big(\zeta_{l,0}(a,|x_F|)-\zeta_{(k)}(|(a,x_F)|) \big) \d\Phi_0^{n-k}(v_F,x_F)
	\end{align*}
	for every $a>0$ and $v_F\in\fconvfsF$. In particular, this implies that
	\begin{multline*}
	\int_E \int_F \zeta_{l,0}(|x_E|,|x_F|) \d \Phi_0^{n-k}(v_F,x_F) \d\Phi_l^{k}(v_E,x_E)\\
	= \int_E \int_F \zeta_{(k)}(\sqrt{|x_E|^2+|x_F|^2}) \d \Phi_0^{n-k}(v_F,x_F) \d\Phi_l^{k}(v_E,x_E)
	\end{multline*}
	for every $v_E\in\fconvfsE$ and $v_F\in\fconvfsF$.

	Note that the remaining case $i=0$ can only occur if $l\leq n-k$. Since by our assumptions $l\geq 2$ and $k\geq 2$ it follows from the previous considerations for the case $i=1$ that $\zeta_{1}=\zeta_{(k)}$. Thus, by Lemma~\ref{le:int_zeta_i_l-i_zeta_i+1_l-i-1} and Lemma~\ref{le:zeta_i},
	\begin{align*}
	    \int_{E'}& \int_{\R} \int_{F'} \zeta_{(k)}(|(x_{E'},a,t,x_{F'})|)\d \Phi_{l-1}^{n-k-1}(v_{F'},x_{F'}) \d t \d x_{E'}
	     \\
	    &= \int_{F'} \int_{\R} \int_{E'} \zeta_{0,l}(|(x_{E'},t)|,|(a,x_{F'})|)
	     \d  x_{E'}  \d t \d \Phi_{l-1}^{n-k-1}(v_{F'},x_{F'}) 
	\end{align*}
	for every $a>0$ and $v_{F'}\in\fconvfsFp$. If we set $v_{F'}(x_{F'})=\frac12(|x_{k+2}-\bar{x}_{k+2}|+\cdots + |x_{k+l}-\bar{x}_{k+l}|)$ with arbitrary $\bar{x}_{k+2},\ldots,\bar{x}_{k+l}\in \R\backslash\{0\}$ and consider \eqref{eq:phi_j_n_delta} this becomes
	\begin{align*}
	    \int_{E'}& \int_{\R} \int_{\R^{n-k-l}} \zeta_{(k)}(|(x_{E'},a,t,b,x_{k+l+1},\ldots,x_n)|) \d x_{k+l+1} \cdots \d x_n \d t  \d x_{E'}
	     \\
	    &= \int_{\R^{n-k-l}} \int_{\R} \int_{E'} \zeta_{0,l}(|(x_{E'},t)|,|(a,b,x_{k+l+1},\ldots,x_n)|) \d x_{E'}  \d t \d x_{k+l+1} \cdots \d x_n
	\end{align*}
	for every $a,b>0$. Thus, by renaming the integration variables on the left side and rearranging, we obtain
	\begin{multline*}
	    \int_{\R^{n-k-l}} \int_{\R} \int_{E'}
	    \Big(\zeta_{(k)}(|(x_{E'},a,t,b,x_{k+l+1},\ldots,x_n)|)\\
	    -\zeta_{0,l}(|(x_{E'},t)|,|(a,b,x_{k+l+1},\ldots,x_n)|) \Big)  \d x_{E'}  \d t \d x_{k+l+1} \cdots \d x_n=0
	\end{multline*}
	for every $a,b>0$ and thus, by Lemma~\ref{lemma Abel} if $l<n-k$ and trivially if $l=n-k$,
	\begin{align*}
	    0&=\int_{\R} \int_{E'} \Big(\zeta_{(k)}(|x_{E'},t,a)|)-\zeta_{0,l}(|x_{E'},t)|,a) \Big) \d x_{E'} \d t\\
	    &= \int_{E} \Big(\zeta_{(k)}(|(x_E,a)|)-\zeta_{0,l}(|x_E|,a) \Big) \d \Phi_0^{k}(v_E,x_E)
	\end{align*}
	for every $a>0$ and $v_E\in\fconvfsE$. In particular, as in the case $i=l-1$, this implies that
	\begin{multline*}
	\int_F \int_E \zeta_{0,l}(|x_E|,|x_F|) \d\Phi_0^{k}(v_E,x_E) \d \Phi_l^{n-k}(v_F,x_F)\\
	= \int_F \int_E \zeta_{(k)}(\sqrt{|x_E|^2+|x_F|^2}) \d\Phi_0^{k}(v_E,x_E) \d \Phi_l^{n-k}(v_F,x_F)
	\end{multline*}
	for every $v_E\in\fconvfsE$ and $v_F\in\fconvfsF$.
\end{proof}

\goodbreak
We can now complete the proof of Proposition \ref{prop:one_zeta}.
	Let 
	$$v(x_1,\ldots,x_n)=\frac 12\left(|x_1-\bar{x}_1|+\cdots+|x_l-\bar{x}_l|\right)$$ with $\bar{x}_1,\ldots,\bar{x}_l\in \R\backslash\{0\}$ and let $a=\sqrt{\bar{x}_1^2+\cdots+\bar{x}_l^2}>0$. Note that for every $\lceil \frac n2 \rceil \leq k \leq n-1$ we can write $v$ in the form $v=v_E+v_F$. Furthermore, by \eqref{eq:phi_i_n_zero}, $\d\Phi_i^k(v_E,x_E)=0$ if $i> l$ and $\d\Phi_j^{n-k}(v_F,x_F)=0$ if $j>(l-k)\vee 0$. Thus, for any $k$, it follows from Lemma~\ref{le:zeta^k}, \eqref{eq:phi_n_n_delta} and \eqref{eq:phi_j_n_delta} that
	\begin{align*}
	\oZ(v)&=\int_{\R^{n-l}} \zeta_{(\lceil \frac n2 \rceil)}(|(a,x_{l+1},\ldots,x_n)|)\d x_{l+1}\cdots \d x_n\\
	&= \cdots\\
	&= \int_{\R^{n-l}} \zeta_{(n-1)}(|(a,x_{l+1},\ldots,x_n)|)\d x_{l+1} \cdots \d x_n.
	\end{align*}
	Hence, by Lemma~\ref{lemma Abel}, $\zeta_{(\lceil \frac n2 \rceil)}\equiv \cdots \equiv \zeta_{(n-1)}=:\zeta$ and $\zeta \in \Had{l}{n}$. Thus, by Lemma~\ref{le:zeta^k},
	\begin{align}
	\begin{split}
	\label{eq:rep_z_v_E_v_F}
    \oZ(v_E+v_F)=&\int_F \int_E \zeta(|(x_E,x_F)|) \d \Phi_0^k(v_E,x_E) \d\Phi_l^{n-k}(v_F,x_F)\\
    &+\sum_{i=1\vee(l+k-n)}^{k\wedge l} \int_E\int_F \zeta(|(x_E,x_F)|) \d\Phi_{l-i}^{n-k}(v_F,x_F)\d\Phi_i^k(v_E,x_E)
    \end{split}
    \end{align}
    for every $v_E\in\fconvfsE$ and $v_F\in\fconvfsF$ and every $\lceil \frac n2 \rceil\leq k \leq n-1$. Since Hessian measures are non-negative and since $v_E+v_F\in\fconvfs$ for every $v_E\in\fconvfsE$ and $v_F\in\fconvfsF$, it follows from Lemma~\ref{lem:existence_v_smooth_origin} and Lemma~\ref{le:decompose_hessian_measure} that
    $$\int_F \int_E \big\vert \zeta(|(x_E,x_F)|) \big\vert \d\Phi_0^k(v_E,x_E) \d\Phi_l^{n-k}(v_F,x_F)\leq \int_{\R^n} \big\vert \zeta(|x|)\big\vert \d\Phi_l^n(v_E+v_F,x)<+\infty$$
    for every $v_E\in\fconvfsE$ and $v_F\in\fconvfsF$. Furthermore, since $\zeta\in\Had{l}{n}$ is measurable and has bounded support and since Hessian measures are locally finite, it now follows from the Fubini-Tonelli theorem that we may exchange the order of integration of the first term on the right side in \eqref{eq:rep_z_v_E_v_F}. Together with Lemma~\ref{le:decompose_hessian_measure} we now obtain
    \begin{align*}
    \oZ(v_E+v_F)&=\sum_{i=0\vee(l+k-n)}^{k\wedge l} \int_E\int_F \zeta(|(x_E,x_F)|) \d\Phi_{l-i}^{n-k}(v_F,x_F)\d\Phi_i^k(v_E,x_E)\\
    &=\int_{\R^n} \zeta(|x|) \d\Phi_l^n(v_E+v_F,x)
    \end{align*}
    for every $v_E\in\fconvfsE$ and $v_F\in\fconvfsF$ and every $\lceil \frac n2 \rceil\leq k \leq n-1$, which completes the proof.
    \bigskip
\subsection*{Acknowledgments}
The authors are grateful to Paolo Salani for his valuable suggestions and, in particular,  for pointing out Lemma \ref{lemma BNST}. Fabian Mussnig has received funding from the European Research Council (ERC) under the European Union's Horizon 2020 research and innovation programme (grant agreement No.~770127).

\bigskip
\end{document}